\documentclass[pdflatex,sn-apa,11pt]{sn-jnl}
\usepackage[shortlabels]{enumitem}
\usepackage{graphicx}%
\usepackage{multirow}%
\usepackage{amsmath,amssymb,amsfonts}%
\usepackage{amsthm}%
\usepackage{mathrsfs}%
\usepackage[title]{appendix}%
\usepackage{xcolor}%
\usepackage{textcomp}%
\usepackage{manyfoot}%
\usepackage{booktabs}%
\usepackage{algorithm}%
\usepackage{algorithmicx}%
\usepackage{algpseudocode}%
\usepackage{listings}%
\usepackage{adjustbox}

\newtheorem{theorem}{Theorem}
\newtheorem*{theoremfive}{Theorem 7}
\newtheorem*{theoremsix}{Theorem 8}
\newtheorem*{theoremnine}{Theorem 9}
\newtheorem*{theoremten}{Theorem 10}
\newtheorem{proposition}{Proposition}

\newtheorem{example}{Example}%
\newtheorem{remark}{Remark}%
\newtheorem{lemma}{Lemma}

\usepackage[pagewise]{lineno}

\begin{document}

\title[Efficiency in a sum of sets]{Efficient points in a sum of sets of alternatives}




\author[]{\fnm{Anas} \sur{Mifrani}\footnote{Toulouse Institute of Mathematics, University of Toulouse, F-31062 Toulouse Cedex 9, France. Email address: anas.mifrani@math.univ-toulouse.fr.\\ ORCID: 0009-0005-1373-9028.}}


\abstract{
The concept of efficiency plays a prominent role in the formal solution of decision problems that involve incomparable alternatives. This paper develops necessary and sufficient conditions for the efficient points in a sum of sets of alternatives to be identical to the efficient points in one of the summands. Some of the conditions cover both finite and infinite sets; others are shown to hold only for finite sets. Examples are provided that illustrate these findings.

}



\maketitle

\section{Introduction}\label{sec1}

Let $G$ denote a set of alternatives and $R$ denote a relation for comparing pairs of alternatives. For every $x$ and $y$ in $G$, $xRy$ shall have the interpretation that $x$ is at least as good as $y$. Where $x$ is better than $y$, we shall write $xPy$, indicating that $xRy$ holds but that $yRx$ does not. If neither $xRy$ nor $yRx$ obtains for some $x, y \in G$, we say $x$ and $y$ are incomparable, and write $xIy$ or $yIx$. 

We shall concern ourselves with the set equality
\begin{equation}\tag{E}\label{E}\mathscr{E}(A + B) = \mathscr{E}(A),\end{equation}
where $A$ and $B$ are nonempty subsets of $G$, $A + B = \{a + b: \forall a \in A, \forall b \in B\}$ is the sum set (or Minkowski sum) of the two, and $\mathscr{E}(S)$ is the set of all \textit{efficient} points in a set $S \subseteq G$. We say a point $g \in G$ is efficient in $S$ if $g \in S$ and no $g' \in S$ exists such that $g'Pg$. The concept of efficiency plays a prominent role in economics, game theory, statistical decision theory, as well as in the formal analysis and solution of multiple objective optimization problems \citep{GEOFFRION1968618, benson1978existence, ganjali2024multi, mifrani2025linear}.

A special case of this equality appears in two influential monographs in the field of multiple criteria decision analysis. According to \cite{yu1985multiple} and \cite{ehrgott2005multicriteria}, when $A$ is an arbitrary subset of $\mathbb{R}^{q}$ and $B$ is the nonpositive orthant of $\mathbb{R}^{q}$ in the sense of the usual, ``larger-is-better", product order, which is given by $xRy \iff x_i \geqq y_i$ for all $i = 1, ..., q$ and $x, y \in \mathbb{R}^{q}$, equality (\ref{E}) holds. Stated otherwise, if $A$ is shifted by nonpositive amounts including zero, then the components of a point in $A$ can only deteriorate or remain intact, so that if the point is efficient in $A$ it is also efficient in the shifted set, and vice-versa. (Yu's justification of this important insight is partly mistaken; we return to this in Section \ref{sec3}.) 

In this paper we shall not be interested in special cases per se, but rather in the kind of conditions that might be imposed on $A$, $B$ and $R$ to guarantee the equality. We shall also not try to work out applications, although it might be pointed out from the outset that the abovementioned Yu-Ehrgott result has helped elucidate the relationship between multiple objective and single-objective optimization problems. Indeed, it is well established that for a mathematical maximization problem involving $q$ criteria functions and a $\Lambda^{\leqq}$-convex value space\footnote{Let $\Lambda^{\leqq}$ denote the nonpositive orthant of $\mathbb{R}^{q}$. A set $Y \subseteq \mathbb{R}^{q}$ is said to be $\Lambda^{\leqq}$-convex if the sum set $Y + \Lambda^{\leqq}$ is convex.}, one consequence of the Yu-Ehrgott version of equality (\ref{E}) is that any efficient solution must maximize a nonnegative linear combination of the $q$ criteria \citep[p. 27]{yu1985multiple}. 
Quite probably other areas exist in which sum sets arise where the equality can serve as a tool for structural investigation of efficiency (as described previously for multiple objective optimization) or as a means of dampening the computational effort required for finding the efficient points in a sum set. We will not be pursuing these ideas here, but we shall briefly return to them in our concluding remarks.  

So far as we can detect, efficiency in sum sets has been the subject of only a handful of publications. \cite{klamroth2024efficient} experiment with various methods for enumerating the efficient points in a finite sum set of vectors. They observe that for certain multiple objective integer programming problems\footnote{A multiple objective integer programming problem is a multiple objective optimization problem in which some or all of the variables are constrained to be integers.}, the efficient value set coincides with the efficient set of the sum of the efficient value sets of smaller, derivative problems, each defined over a distinct subset of the original variables. They also illustrate how the number of efficient points in a finite sum set ``may vary widely", ranging from less than the sum of the summands' cardinalities to the product thereof \cite[Lemma 2]{klamroth2024efficient}. In an earlier article, \cite{kerberenes2023computing} probed two ideas for accelerating the computation of efficient points in a finite sum set of vectors: Exploiting lexicographic orderings on the summands to ``filter" (i.e., test for efficiency) candidates as they are generated, and using ``bounding boxes" to discard whole regions of dominated sums. Computational experience reported in that article indicates that the two approaches may outstrip naive exhaustive search, with the box-based method proving especially effective in the presence of large dominated regions. A similar line of inquiry was pursued recently by \cite{funke2025pareto}, who concentrate on the special case where the summands are themselves efficient sets. It bears mention that outside of this predominantly experimental work, statements about the efficient points of a sum set occurred as lemmata in early developments in the field of multiple criteria decision analysis. We have already mentioned the standard works of \cite{yu1985multiple} and \cite{ehrgott2005multicriteria}, but \cite{moskowitz1975recursion}, for instance, draws on the observation that $\mathscr{E}(A + B) = \mathscr{E}(A + \mathscr{E}(B))$ for certain $A, B \subseteq \mathbb{R}^{q}$ in order to motivate a recursion algorithm for locating efficient decision functions among a finite set of statistical functions. \cite{white1980generalized} rectifies Moskowitz's proof of that identity, shows that it does not hold of any choice of $A$ and $B$ (and $R$), and generalizes it beyond the conditions stipulated by Moskowitz.    


This sampling of the literature gives us reason to believe that this paper carries the first general presentation and treatment of equality (\ref{E}). Section \ref{sec2} describes the assumptions of this work. In Section \ref{sec3}, an examination of the validity of (\ref{E}) under various conditions on $A$, $B$ and $R$ is undertaken. In the course of this examination, it will transpire, for example, that if $R$ is \textit{isotone} (a property to be defined), $\mathscr{E}(A)$ is nonempty, and $B$ contains an element which dominates a certain reference point in $G$, then (\ref{E}) fails (Theorem \ref{thm3}). We adduce examples to buttress the fact that some of the conditions fall short of being simultaneously necessary and sufficient. Our analysis will make it clear that, depending on whether the sets considered are finite or infinite, discrepant conclusions may follow as to the validity of certain theorems. 
We summarize the paper and register some final remarks on the subject in Section \ref{sec4}.     

\section{Notation and assumptions}\label{sec2}

In the subsequent development, we assume $(G, +)$ to be a group of which $A, B$ are nonempty subsets. We do not require $G$ to be abelian, nor that either subset be a subgroup. We let $0_{G}$ represent the identity of $G$. For any nonzero integer $p \geqq 1$ and any $g \in G$, we let $pg$ signify the sum of $p$ copies of $g$ in $G$. When $p = 0$ we set $pg = 0_{G}$. For every $g \in G$, $-g$ shall denote the group inverse of $g$. Efficiency in any subset of $G$ is taken with respect to $R$, and we adopt the convention that $\mathscr{E}(\emptyset) = \emptyset$.

Borrowing slightly from the parlance of \cite{MacLane1999-zu}, we will say that $R$ is isotone if comparisons with respect to it are preserved under addition, that is, if $gRg'$ implies $(z + g)R(z + g')$ and $(g+z)R(g'+z)$ whenever $g, g', z \in G$. 

We assume throughout the exposition that $R$ is a reflexive relation, so that $xRx$ holds of every $x \in G$. Depending on the particular result, we may further assume $R$ to satisfy some or all of the following properties:

\begin{itemize}
\item[(P1)] $R$ is transitive: for every $x, y, z \in G$, if $xRy$ and $yRz$, then $xRz$;
\item[(P2)] $R$ is antisymmetric: for every $x, y \in G$, if $xRy$ and $yRx$, then $x = y$;
\item[(P3)] $R$ is isotone;
\item[(P4)] for each $m \geqq 1$ integers $p_1, \dots, p_m \geqq 1$ and all $b^1, \dots, b^m \in B$, \begin{equation*}(p_1b^1 + \dots + p_mb^{m})R0_{G}\end{equation*} implies that at least one $b^l$ satisfies $b^lR0_{G}$;
\item[(P5)] and for each $m \geqq 1$ integers $p_1, \dots, p_m \geqq 1$ and all $b^1, \dots, b^m \in B$, \begin{equation*}0_{G}R(p_1b^1 + \dots + p_mb^{m})\end{equation*} implies that at least one $b^l$ satisfies $0_{G}Rb^l$;
\end{itemize}

In many of the proofs of Section \ref{sec3}, it will suffice to work with the following simple implications of properties (P4) and (P5):
\begin{itemize}
\item[(P4')] For all $p \geqq 1$ and all $b \in B$, $(pb)R0_{G}$ implies $bR0_{G}$;
\item[(P5')] and for all $p \geqq 1$ and all $b \in B$, $0_{G}R(pb)$ implies $0_{G}Rb$.
\end{itemize}

Note that when any one of properties (P1), (P3), (P4) and (P5) holds of $R$, it also holds of $P$. Reflexivity does not hold of $P$, and nor does (P2). Also observe that the incomparability relation $I$ is symmetric in that $xIy$ implies $yIx$. 

The first two properties are commonplaces of modern utility theory \citep{white1972uncertain, luce1956semiorders}. There has been some debate over the extent to which they accurately capture the preference pattern of a rational decision maker\footnote{Almost any axiom of utility theory has come under attack for this reason; see, for example, \cite{luce1956semiorders} and the references therein. It should be noted, however, that even at the skeptical end of the controversy, commentators like Luce acknowledge the necessity of some of these properties for the development of utility theory-based disciplines. Regarding transitivity, for example, Luce observes that it is ``necessary if one is to have [a] numerical order preserving utility function", and that such functions ``seem indispensable for theories---such as game theory---which rest on preference orderings". He goes on to conclude that ``we too shall take the attitude that, at least for a normative theory, the preference relation should be transitive".}, but this should not be of concern to us. The adequacy of a property of $R$ for approximating real world behavior is too remote a question from the purposes of this paper. Given our wish to explore in as general a setting as possible what in our view is an unexplored theoretical problem, the intuitive grounds on which these properties rest, which are clear enough, are a compelling justification for considering them. 
The only pertinent question, if one is to heed the dictum of the Scholastics in the construction of proper hypotheses\footnote{The dictum, which is sometimes misattributed to William of Ockham, is \textit{Entia non sunt multiplicanda praeter necessitatem}, which translates to ``Entities must not be multiplied beyond necessity". See \cite{beck1943principle} for a philosopher's account of how the dictum has come to be interpreted in the context of the sciences.}, is whether the results reported in this article \textit{necessitate} such properties. We shall not pursue this question here, but we expect it to be the object of future investigation.


The intuitive basis for property (P3) is also clear. A relation satisfying (P1) and (P3) has been termed additive \citep{white1972uncertain}. Properties (P4') and (P5') can be understood as consistency requirements on how $R$ behaves when an element is repeatedly combined with itself. Here we might construe $0_G$ as the ``neutral" alternative. Property (P4') asserts that if several copies of an element $b$ 
are together regarded as being at least as good as $0_G$, then a single copy of $b$ must also enjoy that status. Property (P5') provides the converse safeguard: If $0_G$ is deemed at least as good as a collection of multiple copies of $b$, then it must also be at least as good as a single copy. These properties ensure that $R$ behaves stably 
under scaling and does not admit preference reversals when the same component is repeated.


Some of the proofs given in this paper make no use whatsoever of these properties. Which properties, if any, are used in a proof are indicated in the statement of the apposite proposition. 

Let us now consider our results.

\section{Conditions for the validity of equality (\ref{E})}\label{sec3}



The results developed in this section bear on the domain of validity of equality (\ref{E}). For unencumbered perusal and for an overview of the most important theorems, the reader may wish to consult Table \ref{table1}.

\begin{table}[ht]
\centering
\begin{adjustbox}{max width=\textwidth}
\begin{tabular}{|c|p{2.25cm}|p{2cm}|p{3.5cm}|p{1.75cm}|p{2.75cm}|}
\hline
\textbf{Theorem} & \textbf{$R$} & \textbf{$A$} & \textbf{$B$} & \textbf{$A+B$} & \textbf{Does (E) hold?} \\
\hline
1 & (P3) & $\mathscr{E}(A)=\emptyset$ & n/a & n/a & Yes \\
\hline
2 & (P1)-(P3) & n/a & $0_G \in B$ and $0_GRb$ for all $b \in B$ & n/a & Yes \\
\hline
3 & (P3) & n/a & $\forall b \in B$, $0_GPb$ & $\mathscr{E}(A+B)\neq \emptyset$ & No \\
\hline
4 & (P3) & $\mathscr{E}(A)\neq \emptyset$ & $\exists b^\circ \in B$ with $b^{\circ}P0_G$ & n/a & No \\
\hline
5 & (P2)-(P5) & Finite \& stable & $B=\{b\}$ & n/a & If and only if $b=0_G$ \\
\hline
7 & (P1), (P3), (P4) & Finite & $B=\{b\}$ with $bI0_G$ & n/a & No \\
\hline
8 & (P1), (P3), (P4) & Finite & $B=\{b_1,\ldots,b_m\}$, $m\geqq 2$, $b_iI0_G$ & n/a & No \\
\hline
9 & (P1), (P3), (P5) & Finite \& stable & $B=\{0_G,b\}$ with $bI0_G$ & n/a & No \\
\hline
10 & (P1), (P3), (P5) & Finite \& stable & $B=\{0_G,b_1,\ldots,b_m\}$, $m\geqq 2$, $b_iI0_G$ & n/a & No \\
\hline
\end{tabular}
\end{adjustbox}
\caption{Summary of the principal findings of Section \ref{sec3} when $A+B \neq A$. Columns two through four tabulate the assumptions made on each of $R$, $A$, $B$ and $A+B$ in the corresponding theorem.}
\label{table1}
\end{table}

We begin with two theorems that furnish conditions sufficient for the equality to hold. Theorem \ref{thm0} offers a test of validity that merely involves inspecting the efficient subset of $A$.  

\begin{theorem}
\label{thm0}
Under (P3), $\mathscr{E}(A) = \emptyset$ implies that (\ref{E}) holds.
\end{theorem}
\begin{proof}
Let us assume that $\mathscr{E}(A) = \emptyset$, and suppose, for the sake of contradiction, that $\mathscr{E}(A+B) \neq \emptyset$. We may therefore select $z = a +b \in A+B$, $a \in A$, $b \in B$, such that $z \in \mathscr{E}(A+B)$. Since $A$ has no efficient points, $a \notin \mathscr{E}(A)$, and there exists $a' \in A$ such that $a'Pa$. By (P3), this implies that $(a'+b)Pz$, a contradiction of the fact that $z \in \mathscr{E}(A+B)$. Consequently, $\mathscr{E}(A+B) = \emptyset = \mathscr{E}(A)$.
\end{proof}

\begin{theorem}
\label{thm1}
Under (P1)-(P3), if $0_{G} \in B$ and $0_{G} R b$ whenever $b \in B$, then (\ref{E}) holds.
\end{theorem}

\begin{proof}
If $\mathscr{E}(A) = \emptyset$, then we are done as per Theorem \ref{thm0}. Suppose in what follows that $\mathscr{E}(A) \neq \emptyset$. The proof is in two parts.

(1) 
Let $w^{\circ} \in \mathscr{E}(A)$. Observe that $w^{\circ} = w^{\circ} + 0_{G} \in A + B$, since $w^{\circ} \in A$ and $0_{G} \in B$. Suppose that, for some $w = a + b \in A + B$, $wRw^{\circ}$. Then, from (P1), (P3) and the fact that $0_{G}Rb$, we discover that $aRw^{\circ}$. Because $a \in A$ and $w^{\circ} \in \mathscr{E}(A)$, this means $a = w^{\circ}$. Furthermore, (P3) entitles us to infer that $aRw$ from the statement that $0_GRb$. Therefore, $w^{\circ}Rw$. By (P2), $w = w^{\circ}$. 

Thus we have demonstrated that for all $w \in A+B$, $wRw^{\circ}$ implies $w = w^{\circ}$. Inasmuch as $w^{\circ} \in A + B$, it follows that $w^{\circ} \in \mathscr{E}(A+B)$, so that $\mathscr{E}(A) \subseteq \mathscr{E}(A+B)$. 

(2) If $\mathscr{E}(A+B) = \emptyset$, then $\mathscr{E}(A+B) \subseteq \mathscr{E}(A)$. Otherwise, let $w^{\circ} \in \mathscr{E}(A+B)$ and assume, contrary to the theorem, that $w^{\circ} \notin \mathscr{E}(A)$. Two cases must be considered: $w^{\circ} \in A$, and $w^{\circ} \notin A$. In both cases, we may write $w^{\circ} = a^{\circ} + b^{\circ}$ for some $a^{\circ} \in A$ and $b^{\circ} \in B$.

If $w^{\circ} \in A$, then, in view of the fact that $w^{\circ} \notin \mathscr{E}(A)$, there exists $a \in A$ such that $aRw^{\circ}$ obtains and $w^{\circ}Ra$ does not. However, $a = a + 0_{G}$ being a point in $A + B$, this contradicts the efficiency of $w^{\circ}$ in $A+B$.


If $w^{\circ} \notin A$, then $b^{\circ} \neq 0_{G}$, for the reverse would clearly imply $w^{\circ} \in A$. Now, since $0_{G} R b^{\circ}$, it follows from (P3) that $a^{\circ} R w^{\circ}$. Moreover, we have that $a^{\circ} \neq w^{\circ}$. To see why, suppose the opposite were true. Then, recalling that $G$ is a group and that $a^{\circ}$ admits an inverse element $-a^{\circ}$, we obtain $0_{G} = -a^{\circ} + a^{\circ} = -a^{\circ} + (a^{\circ} + b^{\circ}) = (-a^{\circ} + a^{\circ}) + b^{\circ} = b^{\circ}$, a contradiction. Therefore, $a^{\circ}Rw^{\circ}$ with $a^{\circ} \neq w^{\circ}$. Then, since property (P2) operates, $w^{\circ}Ra^{\circ}$ must not hold. (Otherwise, we would have $a^\circ = w^\circ$.) In conclusion, we have shown that $a^{\circ}Rw^{\circ}$ is true but that $w^{\circ}Ra^{\circ}$ is not. However, this patently violates the presupposed efficiency of $w^\circ$ in $A+B$, as $a^{\circ} = a^{\circ} + 0_{G} \in A+B$.

It follows from the previous two paragraphs that $w^{\circ} \in \mathscr{E}(A)$, and therefore that $\mathscr{E}(A+B) \subseteq \mathscr{E}(A)$.

The combination of parts (1) and (2) of this argument yields the requisite equality.
\end{proof}

\begin{remark}
\label{rem_generalized_equality}
The following extension of Theorem \ref{thm1} can be established fairly easily using induction. Under our hypotheses on $R$, if $B_1, ..., B_n$ are $n$ nonempty subsets of $G$ such that $0_{G} \in B_i$ and $0_{G}Rb_i$ whenever $b_i \in B_i$, for each $i = 1, ..., n$, then \begin{equation*}\mathscr{E}\biggl(A + \sum_{i = 1}^{n}B_i\biggr) = \mathscr{E}(A),\end{equation*}
where $\sum_{i = 1}^{n}B_i = \{b_1 + ... + b_n: b_1 \in B_1, ..., b_n \in B_n\}$. 
\end{remark}

We remarked in Section \ref{sec1} that \cite{yu1985multiple} and \cite{ehrgott2005multicriteria} independently aver a special case of equality (\ref{E}) where $G = \mathbb{R}^{q}$, $R$ is the conventional product order defined in Section \ref{sec1}, $A \subseteq \mathbb{R}^{q}$ and $B = \{d \in \mathbb{R}^{q}: 0Rd\}$. It can be seen from this choice of sets and relation that Yu and Ehrgott's result is in fact a direct corollary of Theorem \ref{thm1}. We must, however, put on record that the proof which Yu proposes is partly fallacious. To demonstrate that $\mathscr{E}(A+B) \subseteq \mathscr{E}(A)$, he---like Ehrgott---assumes as a \textit{reductio ad absurdum} that a point $y^{\circ} \in \mathscr{E}(A + B)$ exists which lies outside $\mathscr{E}(A)$. From this he infers that $y^{\circ}$ must be dominated in $A$: that there must exist $y \in A$ such that $yPy^{\circ}$. But if $y^{\circ} \notin \mathscr{E}(A)$, and all we know is that $y^{\circ} \in \mathscr{E}(A + B)$, then either $y^{\circ} \in A$ and $y^{\circ}$ is indeed dominated in $A$, or $y^{\circ} \notin A$ and the status of $y^{\circ}$ vis-à-vis the elements of $A$ is unclear. Yu overlooks the latter possibility, but nevertheless reaches the same conclusion as Ehrgott, who does take both possibilities into account.

The conditions of Theorem \ref{thm1} are sufficient but not necessary, as highlighted by the next examples. Example \ref{ex1} deals with a case where $A+B=A$, whereas in Example \ref{ex2}, $A+B \neq A$.

\begin{example}
\label{ex1}
Let $G = S_3$ be the symmetric group on $\{1,2,3\}$ under composition. The operation $+$ here denotes permutation composition. Let $A = \{(12),(13)\}$ and $B = \{(123)\}$. Define a relation $R$ on $G$ by
\[
\pi R \sigma \;\; \Longleftrightarrow \;\; 
\pi \text{ has at least as many fixed points as } \sigma.
\]
This reflexive relation fulfills properties (P1) and (P3) to the exclusion of property (P2). To see why property (P2) does not hold of $R$, consider, for example, that $(12)R(13)$ and $(13)R(12)$, yet $(12) \neq (13)$. Notice also that $0_G = \mathrm{id} \notin B$.

However, for each $\pi \in A$, the composition $\pi + (123)$ is again a transposition in $A$
(for instance, $(12)+(123) = (13)$), so that $A+B=A$
and $\mathscr{E}(A+B) = \mathscr{E}(A)$.
\end{example}

\begin{example}
\label{ex2}
Let $G = \mathbb{R}^{2}$, $A = \{(x, y): y = -2x\}$ and $B = \{(-1, 2), (-1, 1)\}$. The relation $R$ is taken to be the standard product order on $\mathbb{R}^{2}$ (see Section \ref{sec1}). Properties (P1)-(P3) are satisfied, but neither condition in Theorem \ref{thm1} obtains. Furthermore, the sum set $A+B$ can easily be shown to be given by \begin{equation*}A + B = A \cup \{(x-1, y+1): (x, y) \in A\}.\end{equation*}
Considering that the sets $A$ and $\{(x-1, y+1): (x, y) \in A\}$ do not intersect (this is readily verifiable), equality (\ref{E}) will follow immediately from the fact, which we shall now elucidate, that for each point $z$ in $\{(x-1, y+1): (x, y) \in A\}$ there is a corresponding point $z'$ in $A$ such that $z'Pz$. Indeed, if $(x, y) \in A$, then $(x-1, -2x+2)P(x-1, y+1)$, with $(x-1, -2x+2) = (x-1, -2(x-1)) \in A$. Consequently, no point in $\{(x-1, y+1): (x, y) \in A\}$ can be efficient in $A + B$, whence 
\begin{equation*}\mathscr{E}(A+B) = \mathscr{E}\biggl(A \cup \{(x-1, y+1): (x, y) \in A\}\biggr) = \mathscr{E}(A).\end{equation*}
\end{example}

The next two theorems delineate situations when equality (\ref{E}) is violated. \textit{In enunciating these theorems we assume $A+B \neq A$} so as to avoid trivialities. When $A+B = A$, the equality obtains irrespective of the conditions imposed on $A$ or $B$.

\begin{theorem}
\label{thm2}
Assume that $\mathscr{E}(A + B)$ is nonempty. Under (P3), if $0_{G}Pb$ for all $b \in B$, then (\ref{E}) fails.
\end{theorem}

\begin{proof}
Let us suppose that $0_{G}Pb$ for all $b \in B$. Let $z = a + b \in A+B$ be an efficient point in $A+B$. The only case of interest is when $z \in A$, because $z \notin A$ implies by definition that $z$ is inefficient in $A$. Let us assume then that $z \in A$. Thanks to property (P3), which also holds of $P$, $aPz$. As a result, $z \notin \mathscr{E}(A)$, and $\mathscr{E}(A + B) \neq \mathscr{E}(A)$.
\end{proof}



\begin{theorem}
\label{thm3}
Assume that $\mathscr{E}(A)$ is nonempty. Under (P3), if there exists a point $b^{\circ} \in B$ such that $b^{\circ}P0_{G}$, then (\ref{E}) fails. 
\end{theorem}

\begin{proof}
Suppose that such a point $b^{\circ}$ exists. We have assumed that $\mathscr{E}(A) \neq \emptyset$. Accordingly, let $a \in A$ be an efficient element in $A$. For the reason explained in the last proof, we need only examine the case where $a \in A+B$. If $a \in A+B$, then, because $a+b^{\circ} \in A+B$ and property (P3) yields $(a+b^{\circ})Pa$, $a$ is inefficient in $A+B$. Thus, $\mathscr{E}(A + B) \neq \mathscr{E}(A)$. 
\end{proof}

Theorems \ref{thm2} and \ref{thm3} may perhaps be usefully rephrased as follows. If $\mathscr{E}(A+B)$ is nonempty, then (\ref{E}) can only be true if there exists a $b^{\circ} \in B$ such that $0_{G}Pb^{\circ}$ does not obtain. Moreover, if $\mathscr{E}(A)$ is nonempty, then (\ref{E}) can only hold if no $b \in B$ satisfies $bP0_{G}$.

In keeping with standard phraseology, we shall call a set \textit{stable} if it is own efficient set and \textit{unstable} otherwise.

\begin{lemma}
\label{lem1}
Under (P2) and (P3), if $A$ is stable and $B$ is a singleton $\{b\}$, $b \in G$, then $A+B$ is stable.
\end{lemma}

\begin{proof}
Let $z, z' \in A+B$ such that $zRz'$, $z = a + b$ and $z' = a' + b$, where $a, a' \in A$. Replicating the calculations involving inverse elements and property (P3) in the proof of Theorem \ref{thm1} produces the conclusion that $aRa'$. We have assumed that $A$ is stable. Therefore, $a' \in \mathscr{E}(A)$ and $a'Ra$ (on account of $aRa'$). By property (P2), $a = a'$ and hence $z = z'$. The reflexivity of $R$ requires that $z'Rz$. To summarize, we have proven that whenever $zRz'$ holds of some $z, z' \in A+B$, $z'Rz$. Ergo, every member of $A+B$ is efficient in $A+B$. 
\end{proof}

\begin{lemma}
\label{lem2}
Under (P2), (P4) and (P5), if $A$ is finite and $B$ is a singleton $\{b\}$, $b \in G$, then $A+B = A$ if and only if $b = 0_{G}$. 
\end{lemma}

\begin{proof}
Write $A = \{a^1, ..., a^n\}$ for $n \geqq 1$ distinct points $a^1, ..., a^n$ in $G$. The subscript $i$ in $a^i$ is not to be confused with a power. The \textit{if} portion of the statement is self-evident. To prove the \textit{only if} portion, we assume that $A+B = A$ and observe that the translation $T: a \mapsto a+b$ defines a one-to-one (surjective) map on $A$. A one-to-one map on a finite set, which $A$ is, happens also to be onto. Hence, $T$ is a permutation of $A$, which means that there exists an integer $p \geqq 1$ such that for each $a \in A$, $a + pb \in A$. For any $a \in A$, therefore, we have that $a + pb = a$, meaning $pb = 0_{G}$. Now, $R$ being reflexive implies that $pbR0_{G}$ and $0_{G}Rpb$. Since properties (P4) and (P5) hold, it follows that $bR0_{G}$ and $0_{G}Rb$, and property (P2) yields the requisite conclusion that $b = 0_{G}$.   

\end{proof}

\begin{remark}
\label{rem_revision_1}
Properties (P4) and (P5), or some equivalent conditions, seem to be vital for Lemma \ref{lem2}, for consider the following counterexample where $R$ is defined by $xRy \iff x=y$. Let $G = \mathbb{Z}/5\mathbb{Z} = \{[0],[1],[2],[3],[4]\}$ be the integers modulo $5$ under addition, $A = G$ and $b = [1]$. The identity element of $G$ is $[0] \neq b$, yet it is plainly the case that $A+\{b\} = G + \{[1]\} = G = A$. Neither property (P4) nor property (P5) (which, incidentally, are equivalent since $R$ is a symmetric relation) holds of $R$, because $5b = [5] = [0]$ and $b\neq [0]$.  

\end{remark}

Taken together, Lemmata \ref{lem1} and \ref{lem2} reveal an interesting feature of equality (\ref{E}). If properties (P2)-(P5) operate, $A$ is a finite stable set and $B$ is a singleton, then equality (\ref{E}) holds if and only if $b = 0_G$, that is, if and only if $A+B=A$.

\begin{theorem}
\label{cor1}
Assume that (P2)-(P5) hold and that $A$ is finite. If $A$ is stable and $B$ is a singleton $\{b\}$, then (\ref{E}) holds if and only if $A+B=A$, if and only if $b = 0_G$. 
\end{theorem}

\begin{remark}
\label{rem_cor1}
In general, neither Theorem \ref{cor1} nor Lemma \ref{lem2} carries over to the case when $A$ has infinitely many points. Indeed, let $G = \mathscr{P}(\mathbb{N})$ be the set of all subsets of the naturals under symmetric difference, where $\mathbb{N} = \{0, 1, ...\}$. The empty set is the identity element of $G$. For any $A, B \subseteq G$, declare $ARB$ if and only if $B \subseteq A$. Let $A = \{\{n\}, \{0, 1, n\}: n = 2, 3,...\}$. By construction, $A$ is stable because none of its elements contain each other. With $B = \{\{0, 1\}\}$, it is not difficult to check that $A+B=A$, and therefore that equality (\ref{E})) holds, despite $\{0, 1\} \neq \emptyset$.   

\end{remark}

Guided by Theorem \ref{thm1}, we have focused our investigation so far on instances of $B$ which contain points equal or comparable to $0_G$. In the spirit of Theorems \ref{thm1}, \ref{thm2} and \ref{thm3}, we might ask what happens when $B$ contains only points incomparable to $0_G$, or when it contains such points \textit{alongside} $0_G$. It should be noted that the answers to these questions cannot appeal to the arguments invoked in the proofs of the previous three theorems, for the centerpiece of these arguments was the ability to compare a point in $A$ with some point in $A+B$ selected to enable the inference that the former, while efficient in $A$, is inefficient in $A+B$ (Theorem \ref{thm3}), or that the latter, while efficient in $A+B$, is inefficient in $A$ (Theorem \ref{thm2}). To see the problem that would arise from applying this approach to the two situations now under study, take just the situation where no points in $B$ compare to $0_{G}$, and suppose one were to proceed as in the proof of Theorem \ref{thm2} or of Theorem \ref{thm3}. In this case, we have that, for each $a \in A$ and $b \in B$, $aI(a+b)$, a fact which in and of itself carries no implications for the efficiency of $a$ or of $a+b$, neither in $A$ nor in $A+B$, even if, as in the proof of Theorem \ref{thm2} or of Theorem \ref{thm3}, one of the two points was known to be efficient.  

It should by now be clear that different ideas are needed for the two situations we have outlined. We devote the remainder of this section to introducing and applying such ideas in the context of a study of the validity of (\ref{E}).

Our chief findings in case all of $B$ is incomparable with $0_{G}$ are Theorems \ref{thm5} and \ref{thm6}. For ease of consumption, we will state the theorems now, deferring their substantiation until enough background material has been presented.

\begin{theoremfive}
Under (P1), (P3) and (P4), if $A$ is finite and $B = \{b\}$ with $bI0_G$, then (\ref{E}) fails.
\end{theoremfive}

\begin{theoremsix}
Under identical hypotheses to those of Theorem \ref{thm5}, if $B = \{b^{1}, ..., b^{m}\}$, $m \geqq 2$, such that $b^{i}I0_G$ for each $i = 1, ..., m$, then (\ref{E}) fails.
\end{theoremsix}

To be sure, Theorem \ref{thm5} is a special case of Theorem \ref{thm6}. However, because the former has a much shorter proof, and because the essentials of the techniques employed in both instances are the same, we will present in this section only the proof of the former, dispatching the more general proof to the appendix. 

The bases for Theorems \ref{thm5} and \ref{thm6} are Theorem \ref{thm4} and Propositions \ref{prop1rep} and \ref{prop1}. Theorem \ref{thm4} is a general result on efficient sets, reported in \cite{white1977kernels} and credited to a theorem in graph theory due to \cite{berge1985graphs}. Its utility here will soon become apparent.

\begin{theorem}[appears in \cite{white1977kernels} as Theorem 3]
\label{thm4}
If $A$ is finite, then, under (P1), the efficient set $\mathscr{E}(A)$ is nonempty, and for every $a \in A$ there exists $a' \in \mathscr{E}(A)$ such that $a'Ra$. In particular, if $a \notin \mathscr{E}(A)$, then $a'Pa$.
\end{theorem}

\begin{remark}In commenting on a special application of Theorem \ref{thm4} in \cite{mifrani2025counterexample}, we suggest that Theorem \ref{thm4} could be interpreted as generalizing the fact that the maximum of a totally ordered set, if it exists, is ``greater" than or equal to any element of that set. The sole efficient point in such a set relative to the order relation is its maximum. Therefore, if we denote this set with $X$, Theorem \ref{thm4} maintains that $\max(X)Rx$ for all $x \in X$.\end{remark}

Equality (\ref{E}) can only be true if all efficient points in $A$ are members of $A+B$. We give two propositions, Propositions \ref{prop1rep} and \ref{prop1}, which show that the hypotheses of Theorem \ref{thm5} preclude this situation. It should be recalled that in Theorem \ref{thm5}, $B = \{b\}$, $bI0_{G}$. Furthermore, since $A$ is finite, $\mathscr{E}(A)$ is nonempty (Theorem \ref{thm4}), and so Theorem \ref{thm0} does not apply. 

\begin{proposition}
\label{prop1rep}
Assume (P3) and (P4). Suppose that $A$ is finite and let $A = \{a^1, ..., a^n\}$, $n \geqq 1$, and $b \in B$. Then, for each set of indices $j_1, ..., j_n$ in $\{1, ..., n\}$, the system
\begin{equation}
\tag{$S^{0}$}
\label{$S^{0}$}
  \left\{
    \begin{aligned}
      & a^1 = a^{j_1} + b, \\
      & \dots \\
      & a^n = a^{j_n} + b, \\
      & bI{0_G},
    \end{aligned}
  \right.
\end{equation}
is inconsistent.
\end{proposition}

\begin{proposition}
\label{prop1}
Assume (P1), (P3) and (P4). Suppose $A$ is a finite unstable set, so that it contains at least two points. Let $A = \{a^1, ..., a^n\}$, $n \geqq 2$, and $b \in B$. Then, for each $k = 1, ..., n-1$, for each set of indices $j_1, ..., j_k$ in $\{1, ..., n\}$, and for each set of indices $i_{k+1}, ..., i_{n}$ in $\{1, ..., k\}$, the system
\begin{equation}
\tag{$S^{1}$}
\label{$S^{1}$}
  \left\{
    \begin{aligned}
      & a^1 = a^{j_1} + b, \\
      & \dots \\
      & a^k = a^{j_k} + b, \\
      & a^{i_{k+1}}Pa^{k+1}, \\
      & \dots \\
      & a^{i_{n}}Pa^{n}, \\
      & bI{0_G},
    \end{aligned}
  \right.
\end{equation}
is inconsistent.
\end{proposition}

\begin{remark}
These propositions and the two theorems which derive from them are of interest primarily when $A$ contains more than one point. The case of a singleton $A$ falls within the scope of Theorem \ref{cor1}. One implication of that theorem is that if $A = \{a^{1}\}$, $a^{1} \in G$, and $bI0_{G}$, then $\mathscr{E}(A+B) \neq \mathscr{E}(A)$.    
\end{remark}

To gain some insight into the significance of these propositions, one should think of $a^1, ..., a^{k}$ as the efficient points of $A$. Either these constitute the whole of $A$ or they do not. If $\mathscr{E}(A) = A$, Proposition \ref{prop1rep} tells us that, since $bI0_{G}$, at least one of the points must lie outside $A+B$, hence $A = \mathscr{E}(A) \not \subseteq A+B$. If not, the points $a^{k+1}, ..., a^{n}$ represent the inefficient portion of $A$, and for each $j = k+1, ..., n$ Theorem \ref{thm4} provides us with an index $i_j = 1, ..., k$ satisfying $a^{i_j}Pa^{j}$. Proposition \ref{prop1} merely states that, because $bI0_{G}$, the previous sentence is incompatible with a situation in which all of the $a^{j}$, $i = 1, ..., k$, were in $A+B$, thus implying that $\mathscr{E}(A) \not \subseteq A+B$ as before.  

The proof that system (\ref{$S^{0}$}) is inconsistent is somewhat uncomplicated. Take any set of indices $j_1, ..., j_n$ from $\{1, ..., n\}$. In the event that an index $j_i$ equalled $i$, the $i$-th equation in (\ref{$S^{0}$}) would imply $b = 0_{G}$ by virtue of property (P3), and that would preclude $bI0_{G}$. (This takes care of the case $n = 1$.) If, on the other hand, $n \geqq 2$ and the indices were selected in such a way that $j_i \neq i$ for each $i$, a simple combinatorial argument will yield the sought inconsistency result. We make such an argument below.

\begin{proof}[Proof of Proposition \ref{prop1rep}]

Some specific terminology is in order. Within system (\ref{$S^{0}$}), we shall identify as a \textit{cycle} any group of $r \leqq n$ equations that bears the form
\begin{equation*}
  \left\{
    \begin{aligned}
      & a^{l_1} = a^{l_2} + b, \\
      & a^{l_2} = a^{l_3} + b, \\
      & \dots \\
      & a^{l_r} = a^{l_1} + b,
    \end{aligned}
  \right.
\end{equation*} 
with $l_1, \dots, l_r$ being integers among $1, \dots, n$. Our argument is that (\ref{$S^{0}$}) necessarily exhibits a cycle. 

Let $f$ be the permutation that sends each $i = 1, \dots, n$ to the corresponding right-hand side index $j_i$. Fix some reference index $i_0$ and form the sequence \begin{equation*}i_0, f(i_0), \dots, f^{n}(i_0).\end{equation*} This sequence says that \begin{equation*}
  \left\{
    \begin{aligned}
      & a^{i_0} = a^{f(i_0)} + b, \\
      & a^{f(i_0)} = a^{f^2(i_0)} + b, \\
      & \dots \\
      & a^{f^{n-1}(i_0)} = a^{f^{n}(i_0)} + b,
    \end{aligned}
  \right.
\end{equation*}
Now each element in the sequence ranges in $\{1, \dots, n\}$, while the sequence itself runs to $(n+1)$ members. According to the pigeonhole principle, therefore, there must exist a pair of integers $0 \leqq u < m$ with $f^{u}(i_0) = f^{m}(i_0)$. Set $r = m-u$. The subsequence starting at $l_1 := f^{u}(i_0)$ and culminating in $l_r := f^{m}(i_0) = l_1$ engenders a cycle in the sense that the equations\begin{equation*}
  \left\{
    \begin{aligned}
      & a^{l_1} = a^{l_2} + b, \\
      & a^{l_2} = a^{l_3} + b, \\
      & \dots \\
      & a^{l_r} = a^{l_1} + b,
    \end{aligned}
  \right.
\end{equation*}
hold. We conclude from this that $a^{l_1} = a^{l_1} + rb$, which is to say $0_G = rb$. But if property (P4) operates, this can only be true if $bR0_G$.

\end{proof}

With regard to Proposition \ref{prop1}, one can see that system (\ref{$S^{1}$}) depends on $k$ as well as on the choice of $i_1, ..., i_k$ and $j_{k+1}, ..., j_n$. We omit this dependency in order to simplify notation, but it ought always to be borne in mind. We believe it helps, as a prelude to proving this proposition, to offer some corroborating examples in Examples \ref{ex4}, \ref{ex5} and \ref{ex6}. The treatment of these simple and, as we shall soon see, carefully selected examples will shed light on the basic strategy of our proof.  

\begin{example}
\label{ex4}
Let $A = \{a^1, a^2, a^3, a^4, a^5\}$, $b \in G$ and $k = 3$. Suppose $a^1, ..., a^5$ and $b$ satisfy
\begin{equation*}
  \left\{
    \begin{aligned}
      & a^1 = a^{2} + b, \\
      & a^2 = a^{4} + b, \\
      & a^3 = a^{5} + b, \\
      & a^{3}Pa^{4}, \\
      & a^{1}Pa^{5}, \\
      & bI{0_G},
    \end{aligned}
  \right.
\end{equation*}
The two comparisons involving $P$ can be rewritten as $(a^{5}+b)Pa^{4}$ and $(a^{4}+2b)Pa^{5}$. By (P3), we can again rewrite the second comparison as $(a^{4}+3b)P(a^{5}+b)$. By (P1), then, $(a^{4}+3b)Pa^{4}$. This last fact implies $(3b)P0_{G}$, which in turn implies $bP0_{G}$ due to (P4). This conclusion is inconsistent with the fact that $bI0_{G}$. 
\end{example}

\begin{example}
\label{ex5}
Take the same setting as in Example \ref{ex4} and substitute $a^{2}Pa^{4}$ for $a^{3}Pa^{4}$. Since $a^2 = a^4 + b$, this comparison directly implies that $bP0_{G}$, again in contradiction with $bI0_G$. 
\end{example}

\begin{example}
\label{ex6}
Let $A = \{a^1, a^2, a^3, a^4, a^5, a^6\}$, $b \in G$ and $k = 3$. Suppose $a^1, ..., a^6$ and $b$ satisfy
\begin{equation*}
  \left\{
    \begin{aligned}
      & a^1 = a^{4} + b, \\
      & a^2 = a^{5} + b, \\
      & a^3 = a^{6} + b, \\
      & a^{2}Pa^{4}, \\
      & a^{3}Pa^{5}, \\
      & a^{1}Pa^{6}, \\
      & bI{0_G},
    \end{aligned}
  \right.
\end{equation*}
The equations allow us to rewrite the comparisons as $(a^5 + b)Pa^4$, $(a^6+b)Pa^5$ and $(a^4 + b)Pa^{6}$. The import of the first two comparisons is that $(a^6 + 2b)Pa^{4}$. By combining the latter comparison with $(a^4 + b)Pa^6$ we derive the further comparison $(a^4 + 3b)Pa^{4}$, hence $(3b)P0_G$ and therefore $bP0_G$. However, that disagrees with the fact that $bI0_G$. 
\end{example}

These examples suggest at once a procedure for demonstrating the inconsistency of system (\ref{$S^{1}$}) in general. First, we eliminate $a^1, ..., a^k$ from the comparisons by replacing them with the right-hand side expressions provided by the $k$ equations. If a comparison of the form $(a^{i}+pb)Pa^{i}$ is revealed, we are done (Example \ref{ex5}), and if a pair of comparisons of the form $(a^{i}+pb)Pa^{j}$ and $(a^{j}+p'b)Pa^{i}$, $j \neq i$, appear instead, we are also done (Example \ref{ex4}). If neither situation arises, we generate, following Example \ref{ex6}, a new comparison $(a^{i}+pb)Pa^{j}$ with $j \neq i$, then we search the original set of comparisons for one of the form $(a^{j}+p'b)Pa^{i}$. If such a comparison exists, we are done; otherwise we continue generating comparisons as indicated until this situation obtains, taking at each iteration the preceding set of comparisons as our point of departure\footnote{Readers trained in linear programming will notice a slight resemblance between this procedure and Fourier's procedure \citep{williams1986fourier} for eliminating variables from a system of linear inequalities. For example, when in the final step of Example \ref{ex4} we inferred that $(a^{4}+ 3b)Pa^{4}$, we, in Fourier's language, eliminated the ``variable" $a^{6}$ between the two ``inequalities" $a^{6}P(a^{4} + (-2b))$ and $(a^{4} + b)Pa^{6}$.}.

We must attend to a couple of background issues if we are to operationalize these steps as a method of proof. In the first place, although none of the examples bears this out, it is sometimes impossible to eliminate from the comparisons all of $a^1, ..., a^k$. Consider a variation of Example \ref{ex4} where the equations read successively $a^1 = a^2 + b$, $a^2 = a^1 + b$ and $a^3 = a^5 + b$, subject to the same comparisons as before. It is clear that neither $a^1$ nor $a^2$ can be discarded from the comparison $a^1Pa^5$, and so the first step of the procedure fails. In the second place, it is not immediately obvious why, if no comparison of the form $(a^i + pb)Pa^{i}$ is present initially, there must exist a pair of comparisons of the form specified above, either in the original system or in the sequence of systems produced by the transformation illustrated in Example \ref{ex6}.  

In light of these problems it is useful to recall cycles as introduced in the proof of Proposition \ref{prop1rep}. Given some $k = 1, ..., n-1$ and some choice of associated indices, we shall classify as a cycle of (\ref{$S^{1}$}) any equation of the form $a^i = a^i + b$, $i = 1, ..., k$, and any group of $r \leqq k$ equations of the form
\begin{equation*}
  \left\{
    \begin{aligned}
      & a^{l_1} = a^{l_2} + b, \\
      & a^{l_2} = a^{l_3} + b, \\
      & \dots \\
      & a^{l_r} = a^{l_1} + b,
    \end{aligned}
  \right.
\end{equation*}where the superscripts $l_1, \dots, l_r$ range in $\{1, \dots, k\}$. A cycle is formed, for example, by the equations $a^{1} = a^{2} + b$ and $a^2 = a^1 + b$ of the above-described variant of Example \ref{ex4}.

Cycles need not exist for a particular choice of $k$, of $i_1, ..., i_k$ and of $j_{k+1}, ..., j_n$, but it is easy to show (see the proof of Proposition \ref{prop1rep}) that when they do, we will find at least one index $i = 1, ..., k$ and an integer\footnote{In reality, we will have that $a^{i} = a^{i} + pb$ for \textit{any} $p \geqq 1$.} $p \geqq 1$ such that $a^{i} = a^{i} + pb$, from which it will follow that $pb = 0_{G}$. But if, in addition to $R$ being reflexive, property (P4) operates, this will mean that $bR0_{G}$, whereas in fact $bI0_{G}$. 

This exhausts the treatment of cyclical systems. If (\ref{$S^{1}$}) is not cyclical, then Lemma \ref{prop2} effectively establishes that each of $a^1, ..., a^k$ can be eliminated from the comparisons $a^{i}Pa^{j}$, $i = 1, ..., k$, $j = k+1, ..., n$, as witness Examples \ref{ex4}-\ref{ex6}. 

\begin{lemma}
\label{prop2}
Given an index $k = 1, ..., n-1$, if system (\ref{$S^{1}$}) holds and contains no cycles, then every $a^i$, $i = 1, ..., k$, satisfies the equation $a^i = a^j + pb$ for some $j = k+1, ..., n$ and $p \geqq 1$. As a result, each of $a^1, ..., a^{k}$ is expressible solely in terms of $a^{k+1}, ..., a^{n}$.
\end{lemma}

\begin{proof}
Fix an index $k \in \{1, \dots, n-1\}$, and suppose that system (\ref{$S^{1}$}) holds while exhibiting no cycles. For each index $i \leqq k$, (\ref{$S^{1}$}) provides an equation of the form $a^i = a^{i'} + b$ for some $i' \in \{1, \dots, n\}$.

If $i'$ lies in $\{k+1, \dots, n\}$, then $a^i$ bears the desired representation with $p = 1$ and $j = i'$.

If not, then (\ref{$S^{1}$}) also expresses $a^{i'}$ as $a^{i'} = a^{i''} + b$ for some $i'' \in \{1, \dots, n\}$, and that means $a^i = a^{i''} + 2b$. Carrying on in this fashion for $t$ steps yields the equation $a^i = a^{i_t} + tb$ for an index $i_t \in \{1, \dots, n\}$. There are only two possible outcomes to this iterative process. As there are finitely many indices in $\{1, \dots, k\}$, either we encounter, after finitely many substitutions, an index $i_t$ lying in $\{k+1, \dots, n\}$, in which case we will establish the required representation of $a^i$ as $a^i = a^{i_t} + pb$ with $p = t$, or, alternatively, the process will forever produce indices in $\{1, \dots, k\}$. But should the latter situation arise, the pigeonhole principle assures us that we would have $a^{i_r} = a^{i_s} + (s-r)b$ for some $r < s$---in other words, a cycle. Since we have ruled out the appearance of cycles in (\ref{$S^{1}$}), the process must terminate in the manner described in the first situation.     

\end{proof}

Consider now what happens when this lemma is in force. The $(n-k)$ comparisons that comprise the lower portion of (\ref{$S^{1}$}) could be rewritten as
\begin{equation*}
  \left\{
    \begin{aligned}
      & (a^{m_{k+1}} + p_{k+1}b)Pa^{k+1}, \\
      & \dots \\
      & (a^{m_{n}} + p_{n}b)Pa^{n}, \\
      & bI{0_G},
    \end{aligned}
  \right.
\end{equation*}
where $m_{k+1}, \dots, m_{n} \in \{k+1, \dots, n\}$ and $p_{k+1}, \dots, p_n \geqq 1$ are integers furnished by the lemma. Viewed in this way, the $(n-k)$ comparisons have the curious feature observed in Examples \ref{ex4}-\ref{ex6} that they exhibit a cycle in the following sense. If we let $f(i) = m_{i}$ for each $i = k+1, \dots, n$, it will be contended below that there exist an $r = k+1, \dots, n$ and a sequence of $u \geqq 1$ points $a^{r}, a^{f(r)}, \dots, a^{f^{u-1}(r)}$ such that 
\begin{equation*}
  \left\{
    \begin{aligned}
      & (a^{f(r)} + p_{r}b)Pa^{r}, \\
      & (a^{f^2(r)} + p_{f(r)}b)Pa^{f(r)}, \\
      & \dots \\
      & (a^{r} + p_{f^{u-1}(r)}b)Pa^{f^{u-1}(r)}, \\
      & bI{0_G},
    \end{aligned}
  \right.
\end{equation*}
from which it will follow that $(a^{r} + (p_{f^{u-1}(r)} + \dots + p_{f(r)} + p_{r})b)Pa^{r}$ by properties (P1) and (P3), then $(p_{f^{u-1}(r)} + \dots + p_{f(r)} + p_{r})bP0_G$ by property (P3), then, finally, $bR0_G$ by property (P4). 

With this as background, we are now in a position to demonstrate Proposition \ref{prop1}.

\begin{proof}[Proof of Proposition \ref{prop1}]
The observation that $A$ contains at least two points as a result of its instability springs from $P$'s irreflexivity.

Let $k = 1, ..., n-1$. We need not concern ourselves with the case where (\ref{$S^{1}$}) exhibits a cycle, for the discussion preceding Lemma \ref{prop2} dealt with this contingency. Rather, we assume no cycles exist. According to Lemma \ref{prop2}, therefore, there exist some indices $m_{k+1}, ..., m_{n}$ in $\{k+1, ..., n\}$ and some integers $p_{k+1}, ..., p_{n} \geqq 1$ for which
\begin{equation*}
  \left\{
    \begin{aligned}
      & (a^{m_{k+1}} + p_{k+1}b)Pa^{k+1}, \\
      & \dots \\
      & (a^{m_{n}} + p_{n}b)Pa^{n}, \\
      & bI{0_G}.
    \end{aligned}
  \right.
\end{equation*}

If $k+1 = n$, then $m_n = n$ and we are done, as Example \ref{ex5} and the ensuing discussion make clear. Let us assume henceforth that $k+1 < n$, which presupposes that $n \geqq 3$. (If $n = 2$, $k+1$ equals $n$ by definition.) 

It is submitted that this system of comparisons, hereafter (SC), is cyclical in the sense outlined in the prelude to this proof. It is also submitted that this fact cannot chime with the assumption that $bI0_G$, hence the system's inconsistency.


For each $i = k+1, \dots, n$, let $f(i) = m_i$; $f$ so defined is a permutation of $\{k+1, \dots, n\}$. Start from any index $i$, and consider the sequence of indices generated by iterative application of $f$: \begin{equation*}i, f(i), f^{2}(i), \dots\end{equation*} 
Because the terms of this sequence run over a finite set, there must be at least two of them, say $f^{r}(i)$ and $f^{m}(i)$ with $0 \leqq r \leqq m$, that are identical. The subsequence that spans $f^{r}(i), \dots, f^{m}(i)$ says that \begin{equation*}
  \left\{
    \begin{aligned}
      & (a^{f^{r+1}(i)} + p_{f^{r}(i)}b)Pa^{f^{r}(i)}, \\
      & \dots \\
      & (a^{f^{m}(i)} + p_{f^{m-1}(i)}b)Pa^{f^{m-1}(i)}. \\
    \end{aligned}
  \right.
\end{equation*}
Alternating application of transitivity (P1) and isotonicity (P3) to these $(m-r)$ comparisons (beginning with the bottom comparison and ascending) generates the comparison \begin{equation*}(a^{f^{m}(i)} + (p_{f^{m-1}(i)}+\dots + p_{f^{r}(i)})b)Pa^{f^{r}(i)},\end{equation*}
which is to say \begin{equation*}(p_{f^{m-1}(i)}+\dots + p_{f^{r}(i)})bP0_G,\end{equation*}
because $f^{m}(i) = f^{r}(i)$ and property (P3) holds. Thus $bR0_G$ by property (P4). Therefore, (SC) is inconsistent, and a fortiori so is (\ref{$S^{1}$}).
\end{proof}


From Propositions \ref{prop1rep} and \ref{prop1} and Theorem \ref{thm4} follows Theorem \ref{thm5}.

\begin{theorem}
\label{thm5}
Under (P1), (P3) and (P4), if $A$ is finite and $B = \{b\}$ with $bI0_G$, then (\ref{E}) fails.
\end{theorem}

\begin{proof}
    Write $A = \{a^1, ..., a^n\}$, $n \geqq 2$. After relabeling the points if necessary, write $\mathscr{E}(A) = \{a^1, ..., a^k\}$, with $k = 1, ..., n$. 
    
    If $A$ is stable, then $\mathscr{E}(A) = A \not\subseteq A+B$ by Proposition \ref{prop1rep}, from which we conclude that $\mathscr{E}(A) \neq \mathscr{E}(A+B)$. 
    
    If $A$ is unstable, then $k < n$, and we let $D(A) = \{a^{k+1}, ..., a^{n}\}$ denote the dominated portion of $A$. Since $A$ is finite, Theorem \ref{thm4} assures us that for each $j = k+1, ..., n$ there exists $i_j = 1, ..., k$ satisfying $a^{i_j}Pa^{j}$. Now, if $\mathscr{E}(A+B) = \mathscr{E}(A)$ were true, we would have that $\mathscr{E}(A) \subseteq A+B$, i.e.,
    \begin{equation*}
  \left\{
    \begin{aligned}
      & a^1 = a^{j_1} + b, \\
      & \dots \\
      & a^k = a^{j_k} + b, \\
    \end{aligned}
  \right.
\end{equation*}
for some $k$ indices $j_1, ..., j_k$ in $\{1, ..., n\}$. But Proposition \ref{prop1} tells us that cannot be so, as we already have $bI0_G$ and $a^{i_j}Pa^{j}$ for all $j = k+1, ..., n$. Consequently, $\mathscr{E}(A) \not \subseteq A+B$, and therefore $\mathscr{E}(A) \neq \mathscr{E}(A+B)$.
\end{proof}

If, on top of the notation of Proposition \ref{prop1}, we let $b^1, ..., b^m$ represent $m$ points in $G$ and $s_1, ..., s_m$ represent a set of indices in $\{1, ..., m\}$, then it can be shown that the modified system

\begin{equation}
\tag{$S^{2}$}
\label{$S^{2}$}
  \left\{
    \begin{aligned}
      & a^1 = a^{j_1} + b^{s_1}, \\
      & \dots \\
      & a^k = a^{j_k} + b^{s_m}, \\
      & a^{i_{k+1}}Pa^{k+1}, \\
      & \dots \\
      & a^{i_{n}}Pa^{n}, \\
      & b^{s_i}I{0_G}, \forall i = 1, ..., m,
    \end{aligned}
  \right.
\end{equation}is inconsistent under (P1), (P3) and (P4) for all $k$, $j_1, ..., j_k$ and $i_{k+1}, ..., i_{n}$. Similarly, if, in Proposition \ref{prop1rep}, we substitute any $b^i$ for $b$ in each equation in system (\ref{$S^{0}$}), it poses no particular difficulty to show that the system 
\begin{equation}
\tag{$S^{3}$}
\label{$S^{3}$}
  \left\{
    \begin{aligned}
      & a^1 = a^{j_1} + b^{l_1}, \\
  & \dots \\
      & a^n = a^{j_n} + b^{l_n}, \\
      & b^{l_i}I{0_G},\ \forall i = 1, \dots, m,
    \end{aligned}
  \right.
\end{equation}
is inconsistent for any choice of $l_1, \dots, l_n \in \{1, \dots, m\}$ and $j_1, \dots, j_n \in \{1, \dots, n\}$ under (P3) and (P4).

The proofs are similar to when $B$ contained a single point. To justify the former system's inconsistency we need an extension of Lemma \ref{prop2} wherein the concept of cycle was adjusted to reflect the changes in system (\ref{$S^{1}$}). We also demonstrate such an extension in the appendix (see Section \ref{sec5}).

\begin{lemma}
\label{prop3}
Given an index $k = 1, ..., n-1$, if (\ref{$S^{2}$}) holds and contains no cycles, then every $a^{i}$, $i = 1, ..., k$, satisfies $a^{i} = a^{j} + p_{i, 1}b^{1} + ... + p_{i, m}b^{m}$ for some $j = k+1, ..., n$ and some set of integers $p_{i, 1}, ..., p_{i, m} \geqq 0$. As a result, each of $a^1, ..., a^{k}$ is expressible as a function of and only of $a^{k+1}, ..., a^{n}$.
\end{lemma}

\begin{proof}
See Section \ref{sec5_1}.
\end{proof}

\begin{proposition}
\label{prop_revision_1}
Assume (P1), (P3) and (P4). Suppose $A$ is a finite unstable set, so that it contains at least two points. Let $A = \{a^1, ..., a^n\}$, $n \geqq 2$, and $b_1, \dots, b_m \in G$, $m \geqq 2$. Let $s_1, ..., s_m$ represent a set of indices in $\{1, ..., m\}$. Then, for each $k = 1, ..., n-1$, for each set of indices $j_1, ..., j_k$ in $\{1, ..., n\}$, and for each set of indices $i_{k+1}, ..., i_{n}$ in $\{1, ..., k\}$, the system (\ref{$S^{2}$}) is inconsistent.
\end{proposition}

\begin{proof}
See Section \ref{sec5_2}.
\end{proof}

\begin{proposition}
\label{prop_revision_2}
Assume (P3) and (P4). Let $A = \{a^1, ..., a^n\}$, $n \geqq 2$, and $b_1, \dots, b_m \in G$, $m \geqq 2$. Then, for each set of indices $j_1, ..., j_n$ in $\{1, ..., n\}$ and for each set of indices $l_{1}, ..., l_{n}$ in $\{1, ..., m\}$, the system (\ref{$S^{3}$}) is inconsistent.
\end{proposition}

\begin{proof}
See Section \ref{sec5_3}.
\end{proof}

The upshot is that Propositions \ref{prop_revision_1} and \ref{prop_revision_2} and Theorem \ref{thm4} yield Theorem \ref{thm6} in much the same way that Propositions \ref{prop1rep} and \ref{prop1} and Theorem \ref{thm4} yielded Theorem \ref{thm5}.

\begin{theorem}
\label{thm6}
Under identical hypotheses to those of Theorem \ref{thm5}, if $B = \{b^{1}, ..., b^{m}\}$, $m \geqq 2$, such that $b^{i}I0_G$ for each $i = 1, ..., m$, then (\ref{E}) fails.
\end{theorem}

\begin{proof}
See Section \ref{sec5_4}.
\end{proof}

\begin{remark}
\label{rem_thm6}
Theorems \ref{thm5} and \ref{thm6} do not extend to infinite $A$; see Examples \ref{ex1} and \ref{ex2}.
\end{remark}

We turn, finally, to the case when $0_G \in B$ and $bI0_{G}$ for all $b \in B \backslash \{0_G\}$. If $0_{G} \in B$, then, contrary to what we have seen before, all efficient points in $A$ must belong to $A+B$. Our main conclusions, which we shall substantiate presently, can be stated as follows.

\begin{theoremnine}
Suppose $A$ is finite. Under (P1), (P3) and (P5), if $A$ is stable and there exists $b \in G$ such that $B = \{0_{G}, b\}$ and $bI0_{G}$, then (\ref{E}) fails.
\end{theoremnine}

\begin{theoremten}
Under identical hypotheses to those of Theorem \ref{thm7}, if $A$ is stable and $B = \{0_{G}, b^{1}, ..., b^{m}\}$, $m \geqq 2$, where $b^{i}I0_{G}$ for each $i = 1, ..., m$, then (\ref{E}) fails.
\end{theoremten}

Remarks \ref{rem_thm7} and \ref{rem_thm8} will underscore that these theorems do not generalize to infinite $A$.

We begin with two important preliminaries. Proposition \ref{prop4} presents a logical inconsistency result akin to Propositions \ref{prop1rep} and \ref{prop1}. 

\begin{proposition}
\label{prop4}
Assume (P1), (P3) and (P5). Let $a^1, ..., a^n$ be $n \geqq 2$ distinct points in $G$, and let $b \in B$. Form a set of indices $i_1, ..., i_n$ from $\{1, ..., n\}$ with $i_j \neq j$ for each $j$. Then the system
\begin{equation}
\tag{$S^{4}$}
\label{$S^{4}$}
  \left\{
    \begin{aligned}
      & a^{i_1}P(a^{1} + b), \\
      & \dots \\
      & a^{i_{n}}P(a^{n} + b), \\
      & bI{0_G},
    \end{aligned}
  \right.
\end{equation}
is inconsistent.
\end{proposition}

\begin{proof}
Suppose system (\ref{$S^{4}$}) were consistent. Define the permutation $f:\{1, \dots, n\} \rightarrow \{1, \dots, n\}$ that maps each $j = 1, \dots, n$ to the index $i_j$ which appears on the left-hand side of the comparison $a^{i_j}P(a^{j}+b)$. Fix any index $j_0$ and consider the sequence \begin{equation*}j_0, j_1 := f(j_0), j_2 := f^2(j_0), \dots, j_n:= f^{n}(j_0).\end{equation*}

This sequence says that \begin{equation*}
  \left\{
    \begin{aligned}
      & a^{j_{1}}P(a^{j_0} + b), \\
      & a^{j_{2}}P(a^{j_{1}}+b), \\
      & \dots \\
      & a^{j_n}P(a^{j_{n-1}} + b), \\
    \end{aligned}
  \right.
\end{equation*} 

Now each element in the sequence ranges in $\{1, \dots, n\}$, while the sequence itself runs to $(n+1)$ members. By the pigeonhole principle, therefore, there must exist a pair of integers $0 \leqq r < m$ with $j_r = j_m$. The subsequence starting at $j_r$ and culminating in $j_m$ forms a cycle. Indeed, because this subsequence means \begin{equation*}
  \left\{
    \begin{aligned}
      & a^{j_{r+1}}P(a^{j_r} + b), \\
      & a^{j_{r+2}}P(a^{j_{r+1}}+b), \\
      & \dots \\
      & a^{j_m}P(a^{j_{m-1}} + b), \\
    \end{aligned}
  \right.
\end{equation*}
holds, and inasmuch as $j_m = j_r$, iterative application of transitivity (P1) and isotonicity (P3) to these comparisons (beginning with the bottom comparison and ascending) generates the comparison \begin{equation*}a^{j_r}P(a^{j_r}+(m-r+1)b).\end{equation*} When we apply property (P3) to this very relation, we find that \begin{equation*}0_GP(m-r+1)b,\end{equation*}
so that $0_GPb$ by property (P5). But this cannot be so, since $bI0_{G}$. We infer from this that (\ref{$S^{4}$}) is inconsistent, precisely as the proposition avers. 
\end{proof}

\begin{theorem}
\label{thm7}
Suppose $A$ is finite. Under (P1), (P3) and (P5), if $A$ is stable and there exists $b \in G$ such that $B = \{0_{G}, b\}$ and $bI0_{G}$, then (\ref{E}) fails.
\end{theorem}

\begin{proof}
Assume $B = \{0_{G}, b\}$ for some $b \in G$ incomparable with $0_{G}$, and write $A = \{a^1, ..., a^n\}$ for $n \geqq 1$. 

For the sake of contradiction, let us suppose that equality (\ref{E}) holds. Since $a^i \in A+B$ for each $i = 1, ..., n$ (because $0_{G} \in B$), (\ref{E}) says that all the points in $A+B$ distinct from the $a^i$ are dominated in $A+B$. Such points do exist, because $A+B \neq A$. 
Put otherwise, there exist in $A+B$ at least one pair of indices $i \neq j$ for which $a^{i} \neq a^{j} + b$. Two situations may materialize. 

\begin{itemize}
\item{\textbf{Situation 1:}} $a^{i} \neq a^{j} + b$ for each $i \neq j$. 

Let $z^{i} = a^{i} + b$, $i = 1, ..., n$, be any member of $A+B$. We have assumed here that $z^{i}$ is different than $a^{j}$ for each $j \neq i$. Furthermore, $z^{i} \neq a^{i}$ because $bI0_{G}$. This establishes that $z^{i} \notin \mathscr{E}(A+B) = \mathscr{E}(A)$, and hence, by Theorem \ref{thm4}, the existence of an index $j$ such that $a^{j}Pz^{i}$. Property (P3) tells us that since $bI0_{G}$, the comparison $a^{i}Rz^{i}$---and by extension $a^{i}Pz^{i}$--- does not hold, so that $j \neq i$. Repeating this argument for all $z^{i}$, $i = 1, ..., n$, we construct a set of indices $i_1, ..., i_n$ with $i_j \neq j$ for all $j$ such that 
\begin{equation*}
  \left\{
    \begin{aligned}
      & a^{i_1}P(a^{1} + b), \\
      & \dots \\
      & a^{i_{n}}P(a^{n} + b), \\
    \end{aligned}
  \right.
\end{equation*} 
an absurdity in view of Proposition \ref{prop4}.
\item{\textbf{Situation 2:}} There exist a pair of indices $i^{*} \neq j^{*}$ such that $a^{i^{*}} = a^{j^{*}} + b$.

It will prove useful to work with the set $J = \{j: \forall i = 1, \dots, n,\ a^{i} \neq a^{j}+b\}$, which we argued in the opening paragraph of this proof to be nonempty. For any $j \in J$, $a^{j}+b$ is by construction outside $A$, and so it must lie outside $\mathscr{E}(A+B) = \mathscr{E}(A)$. It follows that for each such $j$ there exists $i_j = 1, \dots, n$ with $a^{i_j}P(a^{j}+b)$. In this situation, we have assumed that there exists an integer (possibly several) $j = 1, \dots, n$ which is not in $J$, and for such an integer we will find an $i_j = 1, \dots, n$ satisfying $a^{i_j}= a^{j} + b$. The net result is a mixed system of comparisons and equalities of the form 
\begin{equation*}
  \left\{
    \begin{aligned}
      & a^{i_{1}}P(a^{j_1} + b), \\
      & \dots \\
      & a^{i_{k}}P(a^{j_k} + b), \\
      & a^{i_{k+1}} = a^{j_{k+1}} + b, \\
      & \dots \\
      & a^{i_{n}} = a^{j_n} + b, \\
    \end{aligned}
  \right.
\end{equation*} 
for some sets of indices $j_1, \dots, j_k \in J$, $k \geqq 1$, $j_{k+1}, \dots, j_n \notin J$ and $i_1, \dots, i_n \in \{1, \dots, n\}$ (with $i_j \neq j$ for each $j$).

We shall show that this system in fact violates the assumption that $bI0_G$. To this end, it helps to borrow some of the ideas of Proposition \ref{prop4}. Let $f$ be the permutation that maps each $j_l$, $l = 1, \dots, n$, to $i_l$. Fix some initial index $j_l$ and consider the sequence of indices \begin{equation*}j_l, f(j_l), \dots, f^{n}(j_l).\end{equation*}
The pigeonhole principle assures us that at least one integer appears twice in this sequence (see the proof of Proposition \ref{prop4} for the full argument), so that there exist $0 \leqq r < m$ for which $f^{r}(j_l) = f^{m}(j_l)$. Denote this integer $u$, and consider the subsequence that begins in $u$ and terminates in $f^{m-r}(u) = u$: \begin{equation*}u, f(u), \dots, f^{m-r-1}(u), u.\end{equation*} 
This subsequence is cyclical in the sense that \begin{equation*}
  \left\{
    \begin{aligned}
      & a^{f(u)}X(a^{u} + b), \\
      & a^{f^2(u)}X(a^{f(u)} + b), \\
      & \dots \\
      & a^{u} X (a^{f^{m-r-1}(u)} + b), \\
    \end{aligned}
  \right.
\end{equation*}
where $X$ signifies either $P$ or $=$, according as the right-hand side index lies (or not) in $J$. Properties (P1) and (P3) operate on both $P$ and $=$, and so it is possible by alternating application of these properties on each of the relations composing the subsystem above---starting, as in the proof of Proposition \ref{prop4}, at the bottom---to arrive at the result that \begin{equation*}a^uX(a^{u} + (m-r+1)b).\end{equation*}
Regardless of whether $X$ here denotes $P$ or $=$, the upshot is that \begin{equation*}0_GXb,\end{equation*}
by virtue of property (P5), which also holds of $=$. However, $bI0_G$ means we can neither have $0_GPb$ nor $0_G = b$.

\end{itemize}
We have seen that either situation leads to a logical contradiction, whence the falsity of the initial premise that $\mathscr{E}(A+B) = \mathscr{E}(A)$.
\end{proof}

\begin{remark}
\label{rem_thm7}
The requirement that $A$ be finite is indispensable to this theorem. For an illustration, modify Example \ref{ex1} by replacing $B$ with the set $\{(0, 0), (-1, 1)\}$. Evidently, $A$ is stable, and each of (P1), (P3) and (P5) holds. However, because we still have $A+B = A$, (\ref{E}) holds.
\end{remark}


\begin{example}
\label{ex7}
To afford the reader with some intuition about Theorem \ref{thm7} and its proof, we will show on a simple example why, if $R$ and $A$ satisfy the hypotheses of Theorem \ref{thm7}, we cannot find a set $B = \{0_{G}, b\}$ where $bI0_{G}$ and for which equality (\ref{E}) holds. 

Let $G = \mathbb{R}^{2}$, $A = \{(-1, 0), (0, -1)\}$ and $B = \{(0, 0), b\}$ with $b = (b_1, b_2)I0_{\mathbb{R}^2}$, all endowed with the standard product order. This order satisfies the hypotheses of Proposition \ref{prop4}, and $A$ is stable as desired.


We ask whether there exist $b_1, b_2$ such that $\mathscr{E}(A+B) = \mathscr{E}(A) = \{(-1, 0), (0, -1)\}$.

Since $(b_1, b_2)I(0, 0)$, we also have $(b_1-1, b_2)I(-1, 0)$ and $(b_1, b_2-1)I(0, 0)$. These shifted points cannot coincide, therefore, with any of the elements of $A$. Consequently, the sum set $A+B$ must take one of three forms: 

\begin{equation*}
A+B = \{(-1, 0), (0, -1), (b_{1}-1, b_{2}), (b_{1}, b_{2}-1)\},
\end{equation*}
or
\begin{equation*}
A+B = \{(-1, 0), (0, -1), (b_{1}, b_{2}-1)\},
\end{equation*}
or
\begin{equation*}
A+B = \{(-1, 0), (0, -1), (b_{1}-1, b_{2})\}.
\end{equation*}

If $\mathscr{E}(A+B) = \mathscr{E}(A) = A$, then each of the cases above necessarily leads to some absurdity. For example, if $A+B$ bears the first form, then it must follow that $(0,-1)P(b_{1}-1, b_2)$ and $(-1,0)P(b_1, b_2-1)$. The former comparison says that $b_1 \leqq 1 $ and $b_2 < 0$, the latter that $b_1 < 0$ and $b_2 \leqq 1$.  
\end{example}

We may suggest the following generalization of Theorem \ref{thm7}. 

\begin{theorem}
\label{thm8}
Under identical hypotheses to those of Theorem \ref{thm7}, if $A$ is stable, and $B = \{0_{G}, b^{1}, ..., b^{m}\}$, $m \geqq 2$, where $b^{l}I0_{G}$ for each $l = 1, ..., m$, then (\ref{E}) fails.
\end{theorem}

\begin{proof}
See Section \ref{sec5_6}.
\end{proof}

\begin{remark}
\label{rem_thm8}
This theorem fails if the assumption that $A$ is finite is eliminated. To take Example \ref{ex2} again, we have, replacing $B$ with $\{(0, 0), (-1, 1), (-1, 2)\}$, that $A$ is stable, that $(-1, 1)I(0, 0)$ and $(-1, 2)I(0, 0)$, that properties (P1), (P3) and (P5) are satisfied, but that \begin{equation*}\mathscr{E}(A+B) = \mathscr{E}\biggl(A \cup \{(x-1, y+1): (x, y) \in A\}\biggr) = \mathscr{E}(A)\end{equation*}
by the same arguments as the original example's. Notice that despite the change in $B$, it is still the case that \begin{equation*}A+B = A \cup \{(x-1, y+1): (x, y) \in A\}.\end{equation*}
\end{remark}

To justify Theorem \ref{thm8} one may proceed by analogy to Theorem \ref{thm7}. Recall that our proof of that theorem drew on Proposition \ref{prop4} and Theorem \ref{thm4}. Theorem \ref{thm4} was invoked for an element of the sum set $A+B$. The only condition for such application is that the target set be finite. Here $A+B$ is still finite, and Theorem \ref{thm4} could therefore be used as is. On the other hand, Proposition \ref{prop4} would have to be broadened to encompass the system
\begin{equation}
\tag{$S^{5}$}
\label{$S^{5}$}
  \left\{
    \begin{aligned}
      & a^{i_{11}}P(a^{1} + b^{1}), \\
      & \dots \\
      & a^{i_{1m}}P(a^{1} + b^{m}), \\
      & a^{i_{21}}P(a^{2} + b^{1}), \\
      & \dots \\
      & a^{i_{2m}}P(a^{2} + b^{m}), \\
      & \dots \\
      & a^{i_{nm}}P(a^{n} + b^{m}), \\
      & b^{j}I{0_G}, \forall j = 1, ..., m,\\
    \end{aligned}
  \right.
\end{equation}where, for each $k = 1, ..., n$ and $j = 1, ..., m$, $i_{kj}$ denotes any integer in $\{1, ..., n\}$ subject to $i_{kj} \neq k$. System (\ref{$S^{5}$}) reduces to (\ref{$S^{4}$}) when $m = 1$. 
In the appendix we show that virtually no techniques beyond those used in connection with the original results are required for accomplishing these generalizations.

\begin{proposition}
\label{prop4_gen}
Assume (P1), (P3) and (P5). 
Let $A=\{a^1,\dots,a^n\}$ with $n\geqq 2$, and let $b^1,\dots,b^m \in G$ with $m\geqq 2$. 
For each $k\in\{1,\dots,n\}$ and each $j\in\{1,\dots,m\}$ fix an index $i_{kj}\in\{1,\dots,n\}$ with $i_{kj}\neq k$. Then system (\ref{$S^{5}$}) is inconsistent.
\end{proposition}

\begin{proof}
See Section \ref{sec5_5}.
\end{proof}



\section{Summary and discussion}
\label{sec4}
The purpose of this paper was to develop conditions for the validity of set equality (\ref{E}). The problem originated in a theorem in \cite{yu1985multiple} and a proposition in \cite{ehrgott2005multicriteria} asserting the equality for specific subsets of $\mathbb{R}^{q}$ and a specific choice of relation. This work was an effort at generalizing (\ref{E}) to situations where the sets considered are arbitrary, possibly non-Euclidean, and where the relation used for comparing alternative points in those sets need not constitute an order relation.

In Theorems \ref{thm0} and \ref{thm1}, we obtained conditions sufficient for (\ref{E}) to hold. Theorem \ref{thm1} subsumes the Yu-Ehrgott result. Theorems \ref{thm2}, \ref{thm3} and \ref{thm5}-\ref{thm8} describe situations in which (\ref{E}) is false. Although this was not made explicit, contraposition of these theorems provides necessary conditions for the validity of the equality. Theorem \ref{cor1} states a necessary and sufficient condition in case $A$ is finite and stable and $B$ contains a single point. 

Theorems \ref{thm2} and \ref{thm3} are each predicated on some efficient set---$\mathscr{E}(A)$ in one case and $\mathscr{E}(A+B)$ in the other---being nonempty. Efficient points are guaranteed to exist in a finite set provided $R$ is transitive (Theorem \ref{thm4}). The situation tends to be more intricate where infinite sets are involved.


A number of promising research programs follow from these conclusions.

Consider first the scope of some of our results. Theorems \ref{thm0}-\ref{thm3} cover both finite and infinite sets. Theorems \ref{thm5}-\ref{thm8} deal strictly with finite sets, while the counterexamples in Remarks \ref{rem_cor1}, \ref{rem_thm6}, \ref{rem_thm7} and \ref{rem_thm8} stress the necessity of this assumption for the four theorems. It remains to be seen whether these theorems have counterparts in cases when $A$ is infinite, $B$ is infinite, or both. This shall be the object of future study.

Next, while the foregoing treatment was entirely theoretical, it would be worthwhile inquiring about its algorithmic implications in light of recent advances in the computation of efficient points in Euclidean sum sets.  \cite{kerberenes2023computing}, \cite{klamroth2024efficient} and \cite{funke2025pareto} devised methods for enumerating the efficient subset of a finite sum set which, across a broad range of experiments, appeared to expend significantly less computation time than exhaustive search. It would be interesting to compare the performance of these methods in determining $\mathscr{E}(A+B)$ with that of a simple enumeration of $\mathscr{E}(A)$ in situations where equality (\ref{E}) holds.  

A further interesting direction would be to extend the analysis beyond groups to, say, semigroups, lattices or cones, where the absence of additive inverses may lead to new conditions for equality validity. A raft of stochastic and combinatorial objects familiar to decision theorists and practitioners---such as lotteries, nonnegative random variables (as typified, for example, by payoffs and costs) and paths on graphs---live in such algebraic structures.

\section{Appendix: Proofs of Lemma \ref{prop3}, Propositions \ref{prop_revision_1}, \ref{prop_revision_2} and \ref{prop4_gen}, and Theorems \ref{thm6} and \ref{thm8}}
\label{sec5}
In Section \ref{sec3} we sketched a program for accomplishing certain extensions of Theorems \ref{thm5} and \ref{thm7} to the case where $0_G \in B$ and $bI0_{G}$ for all $b \in B \backslash \{0_G\}$, but without providing formal proofs. The program revolved around Lemma \ref{prop3} and Propositions \ref{prop_revision_1}, \ref{prop_revision_2} and \ref{prop4_gen}, which have also yet to be demonstrated. In this appendix we furnish the missing proofs.

\subsection{Proof of Lemma \ref{prop3}}
\label{sec5_1}
\begin{proof}

We must preface this proof by adjusting the definition of cycles to system (\ref{$S^{2}$}). We understand a cycle of (\ref{$S^{2}$}) to denote any equation occurring among the system's $k$ equations that has the form $a^i = a^i + b^s$, $i = 1, ..., k$, $s = 1, \dots, m$, and any group of $r \leqq k$ equations occurring among the system's $k$ equations that have the form
\begin{equation*}
  \left\{
    \begin{aligned}
      & a^{l_1} = a^{l_2} + b^{s_{l_1}}, \\
      & a^{l_2} = a^{l_3} + b^{s_{l_2}}, \\
      & \dots \\
      & a^{l_r} = a^{l_1} + b^{s_{l_r}},
    \end{aligned}
  \right.
\end{equation*}where the superscripts $l_1, \dots, l_r$ run over $\{1, \dots, k\}$.  

Fix an index $k \in \{1, \dots, n-1\}$, and suppose that system (\ref{$S^{2}$}) exhibits no cycles in the sense just defined. If the system is consistent, then for each index $i \leqq k$ (\ref{$S^{2}$}) provides an equation of the form $a^i = a^{j} + b^{s}$ for some $j \in \{1, \dots, n\}$ and $s \in \{1, \dots, m\}$. Suppose hereafter that (\ref{$S^{2}$}) is consistent.

If $j$ lies in $\{k+1, \dots, n\}$, then $a^i$ bears the desired representation with $p_{i, s} = 1$ and $p_{i, s'} = 0$ for all $s' \neq s$.

If not, then (\ref{$S^{2}$}) also expresses $a^{j}$ as $a^{j} = a^{j'} + b^{s'}$ for some $j' \in \{1, \dots, n\}$ and $s' \in \{1, \dots, m\}$, and that means $a^i = a^{j'} + b^{s} + b^{s'}$. Carrying on in this fashion for $t$ steps yields the equation $a^i = a^{j_t} + b^{s} + b^{s'} + \dots + b^{s^{(t)}}$ for an index $j_t \in \{1, \dots, n\}$ and some indices $b^{s''}, \dots, b^{s^{(t)}}$ in $\{1, \dots, m\}$. (The superscript $(t)$ in $s^{(t)}$ means the prime symbol $'$ appears $t$ times above $s$.) There are two possible outcomes to this iterative process. As there are finitely many indices in $\{1, \dots, k\}$, either we encounter, after finitely many substitutions, an index $j_t$ lying in $\{k+1, \dots, n\}$, in which case we will have established the required representation of $a^i$, or, alternatively, the process will forever produce indices in $\{1, \dots, k\}$. But should the latter situation arise, the pigeonhole principle assures us that a cycle would form in the infinite sequence of indices generated. Since we have assumed (\ref{$S^{2}$}) to be non-cyclical, the process must terminate in the manner described in the first situation.
\end{proof}

\subsection{Proof of Proposition \ref{prop_revision_1}}
\label{sec5_2}
\begin{proof}
The observation that $A$ contains at least two points as a result of its instability stems from $P$'s irreflexivity.

Let $k = 1, ..., n-1$. We will suppose for the purpose of arriving at a contradiction that (\ref{$S^{2}$}) is consistent. We need not concern ourselves with the case where a cycle appears, for that definitionally implies $b^{s_{l_1}} + \dots + b^{s_{l_r}} = 0_G$, a violation of property (P4) in light of $b^lI0_G$ for each $l = 1, \dots, m$. Rather, we assume no cycles exist. According to Lemma \ref{prop3}, therefore, there exist some indices $m_{k+1}, ..., m_{n}$ in $\{k+1, ..., n\}$ and some integers $p_{i, j} \geqq 0$, $i = k+1, \dots, n$, $j = 1, \dots, m$ for which
\begin{equation*}
  \left\{
    \begin{aligned}
      & (a^{m_{k+1}} + p_{k+1,1}b^1 + \dots + p_{k+1, m}b^m)Pa^{k+1}, \\
      & \dots \\
      & (a^{m_{n}} + p_{n,1}b^1 + \dots + p_{n, m}b^m)Pa^{n}. \\
    \end{aligned}
  \right.
\end{equation*}

If $k+1 = n$, then $m_n = n$ and we are done. Let us assume henceforth that $k+1 < n$, which presupposes that $n \geqq 3$. (If $n = 2$, $k+1$ equals $n$ by definition.) 

The rest of the proof proceeds in the same mould as the proof of Proposition \ref{prop1}. We will refer to the above system as (SC').


Let $f$ denote the permutation of $\{k+1, \dots, n\}$ that maps each $i = k+1, \dots, n$ to $m_i$. Fix some index $i$, and consider the sequence of indices generated through iterative application of $f$: \begin{equation*}i, f(i), f^{2}(i),\dots\end{equation*} 
Because this sequence has its terms in a finite set, there must be at least two of them, say $f^{r}(i)$ and $f^{m}(i)$ with $0 \leqq r \leqq m$, that are identical. The subsequence that spans $f^{r}(i), \dots, f^{m}(i)$ says that \begin{equation*}
  \left\{
    \begin{aligned}
      & (a^{f^{r+1}(i)} + p_{f^{r}(i), 1}b^1 + \dots + p_{f^{r}(i), m}b^m)Pa^{f^{r}(i)}, \\
      & \dots \\
      & (a^{f^{m}(i)} + p_{f^{m-1}(i), 1}b^1 + \dots + p_{f^{m-1}(i), m}b^m)Pa^{f^{m-1}(i)}. \\
    \end{aligned}
  \right.
\end{equation*}
When we apply property (P1) then property (P3) to each of these $(m-r)$ comparisons, beginning with the bottom comparison and ascending, we find that \begin{equation*}\Biggl(a^{f^{m}(i)} + \Biggl(\sum_{l = m-1}^{r}p_{f^{l}(i), 1}\Biggr)b^1 + \dots + \Biggl(\sum_{l = m-1}^{r}p_{f^{l}(i), m}\Biggr)b^m\Biggr)Pa^{f^{r}(i)},\end{equation*}
which is to say \begin{equation*}\Biggl[\Biggl(\sum_{l = m-1}^{r}p_{f^{l}(i), 1}\Biggr)b^1 + \dots + \Biggl(\sum_{l = m-1}^{r}p_{f^{l}(i), m}\Biggr)b^m\Biggr]P0_G,\end{equation*}
because $f^{m}(i) = f^{r}(i)$ and property (P3) holds. Property (P4) entitles us to conclude that there exists at least one $b^l$ among $b^1, \dots, b^m$ satisfying $b^lR0_G$. Therefore, (SC') is inconsistent, and a fortiori so is (\ref{$S^{2}$}).
\end{proof}

\subsection{Proof of Proposition \ref{prop_revision_2}}
\label{sec5_3}
\begin{proof}
Construct, in the manner of Proposition \ref{prop1rep}, a cycle within system (\ref{$S^{3}$}). If the system holds, that cycle would imply that \begin{equation*}0_G = b^{s_1} + \dots + b^{s_r}\end{equation*} for some $r \geqq 1$ indices among $l_1, \dots, l_n$. Because $R$ is symmetric, this would in turn imply \begin{equation*}(b^{s_1} + \dots + b^{s_r})R0_G,\end{equation*}
hence the existence of an $s_i$ such that $b^{s_i}R0_G$, according to property (P4). But that would run counter to the assumption that $b^{s_i}I0_G$.
\end{proof}

\subsection{Proof of Theorem \ref{thm6}}
\label{sec5_4}
\begin{proof}
 Write $A = \{a^1, ..., a^n\}$, $n \geqq 2$. After relabeling the points if necessary, write $\mathscr{E}(A) = \{a^1, ..., a^k\}$, with $k = 1, ..., n$. 
    
    If $A$ is stable, then $\mathscr{E}(A) = A \not\subseteq A+B$ by Proposition \ref{prop_revision_2}. Thus $\mathscr{E}(A) \neq \mathscr{E}(A+B)$. 
    
    If $A$ is unstable, then $k < n$, and we let $D(A) = \{a^{k+1}, ..., a^{n}\}$ denote the dominated portion of $A$. Since $A$ is finite, Theorem \ref{thm4} assures us that for each $j = k+1, ..., n$ there exists $i_j = 1, ..., k$ satisfying $a^{i_j}Pa^{j}$. Now, if $\mathscr{E}(A+B) = \mathscr{E}(A)$ were true, $\mathscr{E}(A)$ would be a subset of $A+B$, i.e.,
    \begin{equation*}
  \left\{
    \begin{aligned}
      & a^1 = a^{j_1} + b^{l_1}, \\
      & \dots \\
      & a^k = a^{j_k} + b^{l_k}, \\
    \end{aligned}
  \right.
\end{equation*}
for some $k$ indices $j_1, ..., j_k$ in $\{1, ..., n\}$ and $k$ indices $l_1, \dots, l_k$ in $\{1, \dots, m\}$. But Proposition \ref{prop_revision_1} tells us that cannot be so, since it is already the case that $b^{l_i}I0_G$ for each $l_i$ and $a^{i_j}Pa^{j}$ for all $j = k+1, ..., n$. Consequently, $\mathscr{E}(A) \not \subseteq A+B$, and therefore $\mathscr{E}(A) \neq \mathscr{E}(A+B)$.
\end{proof}

\subsection{Proof of Proposition \ref{prop4_gen}}
\label{sec5_5}
\begin{proof}
Proposition \ref{prop4_gen} makes the claim that, under properties (P1), (P3) and (P5), if $A=\{a^1,\dots,a^n\}$ with $n\geqq 2$, and $b^1,\dots,b^m \in G$ are $m \geqq 2$ members of $G$, and if for each $k\in\{1,\dots,n\}$ and each $j\in\{1,\dots,m\}$ an index $i_{kj}\in\{1,\dots,n\}$ is selected so that $i_{kj}\neq k$, then the version of system (\ref{$S^{5}$}) engendered by this choice of indices is inconsistent. This conclusion prolongs Proposition \ref{prop4}. 

Notice that when we isolate any of the $b^l$, say $b^m$, a necessary condition for the consistency of system (\ref{$S^{5}$}) is that the subsystem \begin{equation*}
  \left\{
    \begin{aligned}
      & a^{i_{1m}}P(a^{1} + b^m), \\
      & \dots \\
      & a^{i_{nm}}P(a^{n} + b^m), \\
      & b^mI{0_G},
    \end{aligned}
  \right.
\end{equation*} 
be consistent. But it has already been shown (see Proposition \ref{prop4}) that such cannot be the case. Therefore, (\ref{$S^{5}$}) is inconsistent.

\end{proof}

\subsection{Proof of Theorem \ref{thm8}}
\label{sec5_6}
\begin{proof}
As announced in the body of the article, our proof of Theorem \ref{thm8} does not depart in crucial respects from that of its special case Theorem \ref{thm7}.

Assume $B = \{0_{G}, b^1, \dots, b^m\}$ for some $b^1, \dots, b^m \in G$ incomparable with $0_{G}$, and write $A = \{a^1, \dots, a^n\}$ for $n \geqq 1$. 

We will assume, contrary to the theorem, that equality (\ref{E}) holds. Since $a^i \in A+B$ for each $i = 1, ..., n$ (because $0_{G} \in B$), the equality says that all the points in $A+B$ distinct from the $a^i$ are dominated in $A+B$. Such points do exist, because $A+B \neq A$. 
Put otherwise, there exist at least one pair of indices $i \neq j$ in $\{1, \dots, n\}$ and a superscript $l = 1, \dots, m$ for which $a^{i} \neq a^{j} + b^l$. One of two situations may arise. 

\begin{itemize}
\item{\textbf{Situation 1:}} $a^{i} \neq a^{j} + b^l$ for each $i, j = 1, \dots, n$, $i \neq j$, and $l = 1, \dots, m$. 

Let $z^{i, l} = a^{i} + b^{l}$, $i = 1, ..., n$, $l = 1, \dots, m$ be any member of $A+B$. We have assumed here that $z^{i}$ is different than $a^{j}$ for each $j \neq i$. Furthermore, $z^{i} \neq a^{i}$ because $b^{l}I0_{G}$. This establishes that $z^{i} \notin \mathscr{E}(A+B) = \mathscr{E}(A)$, and hence, by Theorem \ref{thm4}, the existence of an index $j$ such that $a^{j}Pz^{i,l}$. Property (P3) tells us that since $b^lI0_{G}$, the comparison $a^{i}Rz^{i,l}$---and by extension $a^{i}Pz^{i,l}$--- does not hold, so that $j \neq i$. By repeating this argument for all $z^{i,l}$, $i = 1, ..., n$, $l = 1, \dots, m$ we construct a set of indices $i_{11}, ..., i_{1m}, \dots, i_{n1}, \dots, i_{nm}$ with $i_{jl} \neq j$ for all $j$ and $l$ such that 
\begin{equation*}
  \left\{
    \begin{aligned}
      & a^{i_{11}}P(a^{1} + b^{1}), \\
      & \dots \\
      & a^{i_{1m}}P(a^{1} + b^{m}), \\
      & a^{i_{21}}P(a^{2} + b^{1}), \\
      & \dots \\
      & a^{i_{2m}}P(a^{2} + b^{m}), \\
      & \dots \\
      & a^{i_{nm}}P(a^{n} + b^{m}), \\
      & b^{j}I{0_G}, \forall j = 1, ..., m,\\
    \end{aligned}
  \right.
\end{equation*}
an absurdity in view of Proposition \ref{prop4_gen}.
\item{\textbf{Situation 2:}} There exist $i^{*} \neq j^{*}$ and $l^* = 1, \dots, m$ such that $a^{i^{*}} = a^{j^{*}} + b^{l^*}$.

It will prove useful to work with the set \begin{equation*}J = \{(j, l): j = 1, \dots, n,\ l = 1, \dots, m,\ \forall i = 1, \dots, n,\ a^{i} \neq a^{j}+b^l\},\end{equation*} which we have argued in the introduction to this proof to be nonempty. Let $k \geqq 1$ denote the cardinality of $J$.

For any $(j,l) \in J$, $a^{j}+b^l$ lies by construction outside $A$ and hence outside $\mathscr{E}(A+B) = \mathscr{E}(A)$. It follows that for each such $(j, l)$ there exists $i_{jl} = 1, \dots, n$ with $a^{i_{jl}}P(a^{j}+b^l)$. In this situation we have assumed that there exist an integer (possibly several) $j = 1, \dots, n$ and an $l = 1, \dots, m$ such that $(j,l)$ is not in $J$, and for such a pair we will find an $i_{jl} = 1, \dots, n$ satisfying $a^{i_{jl}}= a^{j} + b^l$. This leads to a mixed system of comparisons and equalities of the form 
\begin{equation*}
  \left\{
    \begin{aligned}
      & a^{i_{1}}P(a^{j_1} + b^{l_1}), \\
      & \dots \\
      & a^{i_{k}}P(a^{j_k} + b^{l_k}), \\
      & a^{i_{k+1}} = a^{j_{k+1}} + b^{l_{k+1}}, \\
      & \dots \\
      & a^{i_{n \cdot m}} = a^{j_{n\cdot m}} + b^{l_{n \cdot m}}, \\
    \end{aligned}
  \right.
\end{equation*} 
for some pairs of indices $(j_1, l_1), \dots, (j_k, l_k) \in J$, $(j_{k+1}, l_{k+1}), \dots, (j_{n \cdot m}, l_{n \cdot m}) \notin J$ and $i_1, \dots, i_{n \cdot m} \in \{1, \dots, n\}$ (with $i_j \neq j$ for each $j$). Notice that because $J$ and its complement partition the product set $\{1, \dots, n\} \times \{1, \dots, m\}$, all of the latter's elements appear on the right-hand side of the system. 

We shall show that this system cannot be reconciled with the assumption that $b^lI0_G$ for each $l = 1, \dots, m$. Let $f$ map each $(j_p, l_p)$ in the system, $p = 1, \dots, n \cdot m$, to $i_p$. In this connection, denote with $g$ the mapping that sends each $(j_p, l_p)$ to $f(f(j_p, l_p), 1) \in \{1, \dots, n\}$. 

Fix some arbitrary pair $(j_p, l_p)$ and consider the sequence \begin{equation*}f(j_p, l_p), g(j_p, l_p), \dots,  g^{n \cdot m}(j_p, l_p).\end{equation*}
The import of this sequence is that 
\begin{equation*}
  \left\{
    \begin{aligned}
      & a^{f(j_p, l_p)}X(a^{j_p} + b^{l_p}), \\
      & a^{g(j_p, l_p)}X(a^{f(j_p, l_p)} + b^{1}), \\
      & \dots \\
      & a^{g^{n \cdot m}(j_p, l_p)}X(a^{g^{n \cdot m - 1}(j_p, l_p)} + b^{1}), \\
    \end{aligned}
  \right.
\end{equation*} 
where $X$ denotes either $P$ or $=$, according as the right-hand side pair belongs (or not) to $J$. The precise character of each of the $n \cdot m$ relations composing the system is immaterial to our argument.

Each term of the sequence ranges in $\{1, \dots, n\}$, while the total number of terms is $(n \cdot m +1)$. By the pigeonhole principle, at least one term must appear twice in the sequence, and that creates a cycle which, following the argument developed in the proof of Theorem \ref{thm7}, implies \begin{equation*}a^uX(a^{u} + (s-r+1)b^1)\end{equation*}
for some $0 \leqq u \leqq s$. 

Regardless of whether $X$ here denotes $P$ or $=$, we conclude that \begin{equation*}0_GXb^1,\end{equation*}
by virtue of property (P5), which also holds of $=$. This runs athwart the theorem's assumption that $b^lI0_G$ for all $l = 1, \dots, m$. 
\end{itemize}
In both situations, the premise that (\ref{E}) is true leads to a logical contradiction. Thus, the equality fails.

\end{proof}

\backmatter




\section*{Statements and declarations}
This author has no competing interests to declare.

\bibliography{sn-bibliography}

\begin{thebibliography}{}
\renewcommand{\doi}[1]{\url{https://doi.org/#1}}
\bibcommenthead

\bibitem [\protect \citeauthoryear {%
Beck%
}{%
Beck%
}{%
{\protect \APACyear {1943}}%
}]{%
beck1943principle}
\APACinsertmetastar {%
beck1943principle}%
\begin{APACrefauthors}%
Beck, L.W.%
\end{APACrefauthors}%
\unskip\
\newblock
\APACrefYearMonthDay{1943}{}{}.
\newblock
{\BBOQ}\APACrefatitle {The principle of parsimony in empirical science} {The principle of parsimony in empirical science}.{\BBCQ}
\newblock
\APACjournalVolNumPages{The Journal of Philosophy}{40}{23}{617--633,}
\newblock

\newblock

\PrintBackRefs{\CurrentBib}

\bibitem [\protect \citeauthoryear {%
Benson%
}{%
Benson%
}{%
{\protect \APACyear {1978}}%
}]{%
benson1978existence}
\APACinsertmetastar {%
benson1978existence}%
\begin{APACrefauthors}%
Benson, H.P.%
\end{APACrefauthors}%
\unskip\
\newblock
\APACrefYearMonthDay{1978}{}{}.
\newblock
{\BBOQ}\APACrefatitle {Existence of efficient solutions for vector maximization problems} {Existence of efficient solutions for vector maximization problems}.{\BBCQ}
\newblock
\APACjournalVolNumPages{Journal of Optimization Theory and Applications}{26}{}{569--580,}
\newblock

\newblock

\PrintBackRefs{\CurrentBib}

\bibitem [\protect \citeauthoryear {%
Berge%
}{%
Berge%
}{%
{\protect \APACyear {1985}}%
}]{%
berge1985graphs}
\APACinsertmetastar {%
berge1985graphs}%
\begin{APACrefauthors}%
Berge, C.%
\end{APACrefauthors}%
\unskip\
\newblock
\APACrefYear{1985}.
\newblock
\APACrefbtitle {Graphs and Hypergraphs} {Graphs and hypergraphs}.
\newblock
\APACaddressPublisher{Amsterdam}{Elsevier Science Ltd}.
\PrintBackRefs{\CurrentBib}

\bibitem [\protect \citeauthoryear {%
Ehrgott%
}{%
Ehrgott%
}{%
{\protect \APACyear {2005}}%
}]{%
ehrgott2005multicriteria}
\APACinsertmetastar {%
ehrgott2005multicriteria}%
\begin{APACrefauthors}%
Ehrgott, M.%
\end{APACrefauthors}%
\unskip\
\newblock
\APACrefYear{2005}.
\newblock
\APACrefbtitle {Multicriteria optimization} {Multicriteria optimization}.
\newblock
\APACaddressPublisher{}{Springer}.
\PrintBackRefs{\CurrentBib}

\bibitem [\protect \citeauthoryear {%
Funke%
, Hespe%
, Sanders%
, Storandt%
\BCBL {}\ \BBA {} Truschel%
}{%
Funke%
\ \protect \BOthers {.}}{%
{\protect \APACyear {2025}}%
}]{%
funke2025pareto}
\APACinsertmetastar {%
funke2025pareto}%
\begin{APACrefauthors}%
Funke, D.%
, Hespe, D.%
, Sanders, P.%
, Storandt, S.%
\BCBL {} Truschel, C.%
\end{APACrefauthors}%
\unskip\
\newblock
\APACrefYearMonthDay{2025}{}{}.
\newblock
{\BBOQ}\APACrefatitle {Pareto sums of {P}areto sets: Lower bounds and algorithms} {Pareto sums of {P}areto sets: Lower bounds and algorithms}.{\BBCQ}
\newblock
\APACjournalVolNumPages{Algorithmica}{}{}{1--34,}
\newblock

\newblock

\PrintBackRefs{\CurrentBib}

\bibitem [\protect \citeauthoryear {%
Ganjali%
\ \BBA {} Guney%
}{%
Ganjali%
\ \BBA {} Guney%
}{%
{\protect \APACyear {2024}}%
}]{%
ganjali2024multi}
\APACinsertmetastar {%
ganjali2024multi}%
\begin{APACrefauthors}%
Ganjali, N.%
\BCBT {}\ \BBA {} Guney, C.%
\end{APACrefauthors}%
\unskip\
\newblock
\APACrefYearMonthDay{2024}{}{}.
\newblock
{\BBOQ}\APACrefatitle {A Multi-Objective Optimization for Determination of Sustainable Crop Pattern Using Game Theory} {A multi-objective optimization for determination of sustainable crop pattern using game theory}.{\BBCQ}
\newblock
\APACjournalVolNumPages{Journal of Multi-Criteria Decision Analysis}{31}{5-6}{e70000,}
\newblock

\newblock

\PrintBackRefs{\CurrentBib}

\bibitem [\protect \citeauthoryear {%
Geoffrion%
}{%
Geoffrion%
}{%
{\protect \APACyear {1968}}%
}]{%
GEOFFRION1968618}
\APACinsertmetastar {%
GEOFFRION1968618}%
\begin{APACrefauthors}%
Geoffrion, A.M.%
\end{APACrefauthors}%
\unskip\
\newblock
\APACrefYearMonthDay{1968}{}{}.
\newblock
{\BBOQ}\APACrefatitle {Proper efficiency and the theory of vector maximization} {Proper efficiency and the theory of vector maximization}.{\BBCQ}
\newblock
\APACjournalVolNumPages{Journal of Mathematical Analysis and Applications}{22}{3}{618-630,}
\newblock
\begin{APACrefDOI} \doi{https://doi.org/10.1016/0022-247X(68)90201-1} \end{APACrefDOI}
\newblock

\newblock

\PrintBackRefs{\CurrentBib}

\bibitem [\protect \citeauthoryear {%
Kerb{\'e}r{\'e}n{\`e}s%
, Vanderpooten%
\BCBL {}\ \BBA {} Vanpeperstraete%
}{%
Kerb{\'e}r{\'e}n{\`e}s%
\ \protect \BOthers {.}}{%
{\protect \APACyear {2023}}%
}]{%
kerberenes2023computing}
\APACinsertmetastar {%
kerberenes2023computing}%
\begin{APACrefauthors}%
Kerb{\'e}r{\'e}n{\`e}s, A.%
, Vanderpooten, D.%
\BCBL {} Vanpeperstraete, J\BHBI M.%
\end{APACrefauthors}%
\unskip\
\newblock
\APACrefYearMonthDay{2023}{}{}.
\newblock
{\BBOQ}\APACrefatitle {Computing efficiently the nondominated subset of a set sum} {Computing efficiently the nondominated subset of a set sum}.{\BBCQ}
\newblock
\APACjournalVolNumPages{International Transactions in Operational Research}{30}{6}{3455--3478,}
\newblock

\newblock

\PrintBackRefs{\CurrentBib}

\bibitem [\protect \citeauthoryear {%
Klamroth%
, Lang%
\BCBL {}\ \BBA {} Stiglmayr%
}{%
Klamroth%
\ \protect \BOthers {.}}{%
{\protect \APACyear {2024}}%
}]{%
klamroth2024efficient}
\APACinsertmetastar {%
klamroth2024efficient}%
\begin{APACrefauthors}%
Klamroth, K.%
, Lang, B.%
\BCBL {} Stiglmayr, M.%
\end{APACrefauthors}%
\unskip\
\newblock
\APACrefYearMonthDay{2024}{}{}.
\newblock
{\BBOQ}\APACrefatitle {Efficient dominance filtering for unions and {M}inkowski sums of non-dominated sets} {Efficient dominance filtering for unions and {M}inkowski sums of non-dominated sets}.{\BBCQ}
\newblock
\APACjournalVolNumPages{Computers \& Operations Research}{163}{}{106506,}
\newblock

\newblock

\PrintBackRefs{\CurrentBib}

\bibitem [\protect \citeauthoryear {%
Luce%
}{%
Luce%
}{%
{\protect \APACyear {1956}}%
}]{%
luce1956semiorders}
\APACinsertmetastar {%
luce1956semiorders}%
\begin{APACrefauthors}%
Luce, R.D.%
\end{APACrefauthors}%
\unskip\
\newblock
\APACrefYearMonthDay{1956}{}{}.
\newblock
{\BBOQ}\APACrefatitle {Semiorders and a theory of utility discrimination} {Semiorders and a theory of utility discrimination}.{\BBCQ}
\newblock
\APACjournalVolNumPages{Econometrica, Journal of the Econometric Society}{}{}{178--191,}
\newblock

\newblock

\PrintBackRefs{\CurrentBib}

\bibitem [\protect \citeauthoryear {%
MacLane%
\ \BBA {} Birkhoff%
}{%
MacLane%
\ \BBA {} Birkhoff%
}{%
{\protect \APACyear {1999}}%
}]{%
MacLane1999-zu}
\APACinsertmetastar {%
MacLane1999-zu}%
\begin{APACrefauthors}%
MacLane, S.%
\BCBT {}\ \BBA {} Birkhoff, G.%
\end{APACrefauthors}%
\unskip\
\newblock
\APACrefYear{1999}.
\newblock
\APACrefbtitle {Algebra} {Algebra}.
\newblock
\APACaddressPublisher{Providence, RI}{American Mathematical Society}.
\PrintBackRefs{\CurrentBib}

\bibitem [\protect \citeauthoryear {%
Mifrani%
}{%
Mifrani%
}{%
{\protect \APACyear {2025}}%
}]{%
mifrani2025counterexample}
\APACinsertmetastar {%
mifrani2025counterexample}%
\begin{APACrefauthors}%
Mifrani, A.%
\end{APACrefauthors}%
\unskip\
\newblock
\APACrefYearMonthDay{2025}{}{}.
\newblock
{\BBOQ}\APACrefatitle {A counterexample and a corrective to the vector extension of the {B}ellman equations of a {M}arkov decision process} {A counterexample and a corrective to the vector extension of the {B}ellman equations of a {M}arkov decision process}.{\BBCQ}
\newblock
\APACjournalVolNumPages{Annals of Operations Research}{345}{1}{351--369,}
\newblock

\newblock

\PrintBackRefs{\CurrentBib}

\bibitem [\protect \citeauthoryear {%
Mifrani%
\ \BBA {} Noll%
}{%
Mifrani%
\ \BBA {} Noll%
}{%
{\protect \APACyear {2025}}%
}]{%
mifrani2025linear}
\APACinsertmetastar {%
mifrani2025linear}%
\begin{APACrefauthors}%
Mifrani, A.%
\BCBT {}\ \BBA {} Noll, D.%
\end{APACrefauthors}%
\unskip\
\newblock
\APACrefYearMonthDay{2025}{}{}.
\newblock
{\BBOQ}\APACrefatitle {Linear programming for finite-horizon vector-valued {M}arkov decision processes} {Linear programming for finite-horizon vector-valued {M}arkov decision processes}.{\BBCQ}
\newblock
\APACjournalVolNumPages{arXiv preprint arXiv:2502.13697}{}{}{,}
\newblock

\newblock

\PrintBackRefs{\CurrentBib}

\bibitem [\protect \citeauthoryear {%
Moskowitz%
}{%
Moskowitz%
}{%
{\protect \APACyear {1975}}%
}]{%
moskowitz1975recursion}
\APACinsertmetastar {%
moskowitz1975recursion}%
\begin{APACrefauthors}%
Moskowitz, H.%
\end{APACrefauthors}%
\unskip\
\newblock
\APACrefYearMonthDay{1975}{}{}.
\newblock
{\BBOQ}\APACrefatitle {A recursion algorithm for finding pure admissible decision functions in statistical decisions} {A recursion algorithm for finding pure admissible decision functions in statistical decisions}.{\BBCQ}
\newblock
\APACjournalVolNumPages{Operations Research}{23}{5}{1037--1042,}
\newblock

\newblock

\PrintBackRefs{\CurrentBib}

\bibitem [\protect \citeauthoryear {%
White%
}{%
White%
}{%
{\protect \APACyear {1972}}%
}]{%
white1972uncertain}
\APACinsertmetastar {%
white1972uncertain}%
\begin{APACrefauthors}%
White, D.%
\end{APACrefauthors}%
\unskip\
\newblock
\APACrefYearMonthDay{1972}{}{}.
\newblock
{\BBOQ}\APACrefatitle {Uncertain value functions} {Uncertain value functions}.{\BBCQ}
\newblock
\APACjournalVolNumPages{Management Science}{19}{1}{31--41,}
\newblock

\newblock

\PrintBackRefs{\CurrentBib}

\bibitem [\protect \citeauthoryear {%
White%
}{%
White%
}{%
{\protect \APACyear {1977}}%
}]{%
white1977kernels}
\APACinsertmetastar {%
white1977kernels}%
\begin{APACrefauthors}%
White, D.%
\end{APACrefauthors}%
\unskip\
\newblock
\APACrefYearMonthDay{1977}{}{}.
\newblock
{\BBOQ}\APACrefatitle {Kernels of preference structures} {Kernels of preference structures}.{\BBCQ}
\newblock
\APACjournalVolNumPages{Econometrica: Journal of the Econometric Society}{}{}{91--100,}
\newblock

\newblock

\PrintBackRefs{\CurrentBib}

\bibitem [\protect \citeauthoryear {%
White%
}{%
White%
}{%
{\protect \APACyear {1980}}%
}]{%
white1980generalized}
\APACinsertmetastar {%
white1980generalized}%
\begin{APACrefauthors}%
White, D.%
\end{APACrefauthors}%
\unskip\
\newblock
\APACrefYearMonthDay{1980}{}{}.
\newblock
{\BBOQ}\APACrefatitle {Generalized efficient solutions for sums of sets} {Generalized efficient solutions for sums of sets}.{\BBCQ}
\newblock
\APACjournalVolNumPages{Operations Research}{28}{3-part-ii}{844--846,}
\newblock

\newblock

\PrintBackRefs{\CurrentBib}

\bibitem [\protect \citeauthoryear {%
Williams%
}{%
Williams%
}{%
{\protect \APACyear {1986}}%
}]{%
williams1986fourier}
\APACinsertmetastar {%
williams1986fourier}%
\begin{APACrefauthors}%
Williams, H.P.%
\end{APACrefauthors}%
\unskip\
\newblock
\APACrefYearMonthDay{1986}{}{}.
\newblock
{\BBOQ}\APACrefatitle {Fourier's method of linear programming and its dual} {Fourier's method of linear programming and its dual}.{\BBCQ}
\newblock
\APACjournalVolNumPages{The American mathematical monthly}{93}{9}{681--695,}
\newblock

\newblock

\PrintBackRefs{\CurrentBib}

\bibitem [\protect \citeauthoryear {%
Yu%
}{%
Yu%
}{%
{\protect \APACyear {1985}}%
}]{%
yu1985multiple}
\APACinsertmetastar {%
yu1985multiple}%
\begin{APACrefauthors}%
Yu, P\BHBI L.%
\end{APACrefauthors}%
\unskip\
\newblock
\APACrefYear{1985}.
\newblock
\APACrefbtitle {Multiple-Criteria Decision Making: Concepts, Techniques, and Extensions} {Multiple-criteria decision making: Concepts, techniques, and extensions}.
\newblock
\APACaddressPublisher{New York}{Plenum Press}.
\PrintBackRefs{\CurrentBib}

\end{thebibliography}

\end{document}